\documentclass[letterpaper,12pt,reqno]{amsart}
\usepackage{bbm}
\usepackage{amssymb}
\usepackage{amsfonts}
\usepackage{amsmath}
\usepackage[usenames]{color}
\usepackage{mathrsfs}
\usepackage{amsfonts}
\usepackage{amssymb,amsmath}
\usepackage{CJK}
\usepackage{cite}
\usepackage{cases}
\usepackage{amsthm}
\usepackage{setspace}
\usepackage{graphicx}
\usepackage{eucal}
\usepackage{tikz}

\def\la{\lambda}
\def\al{\alpha}
\def\r{\gamma}
\def\b{\beta}

\def\N{\mathbb{N}}

\def\Q{\mathbb{Q}}

\numberwithin{equation}{section}
\newtheorem{theo}{Theorem}[section]
\newtheorem{coro}[theo]{Corollary}
\newtheorem{lemm}[theo]{Lemma}

\theoremstyle{definition}
\newtheorem{defi}[theo]{Definition}

\newtheorem{remark}[theo]{Remark}
\newtheorem{remarks}[theo]{Remarks}

\numberwithin{equation}{theo}

\def\al{\alpha}
\def\r{\gamma}
\def\b{\beta}
\def\l{\lambda}

\def\Z{\mathbb{Z}}

\def\d{\mathbb{D}}
\def\Ajn{{\boldsymbol{\mathfrak{A}}^\jmath(n)}}
\def\fkA{{\mathfrak{A}}}

\def\j{{\jmath}}

\def\ZZ{\mathcal{Z}}
\def\bfU{{\mathbf U}}
\def\up{{v}}
\def\wK{{\widetilde K}}
\def\fS{{\mathfrak S}}

\def\GL{{\text{\rm GL}}}

\def\O{{\text{\rm O}}}

\def\End{{\text{\rm End}}}
\def\sH{{\mathcal H}}
\def\bsH{{\boldsymbol{\mathcal H}}}
\def\sZ{{\mathcal Z}}
\def\sD{{\mathcal D}}
\def\sM{{\mathcal M}}
\def\sS{{\mathcal S}}
\def\sO{{\mathcal O}}
\def\sT{{\mathcal T}}
\def\fkm{{\mathfrak m}}
\def\ro{{\text{\rm ro}}}
\def\co{{\text{\rm co}}}
\def\diag{{\text{\rm diag}}}
\def\Xizero{{\Xi_{2n+1}^{0\diag}}}
\def\wla{{\widetilde\la}}
\def\wmu{{\widetilde \mu}}
\def\wLa{{\widetilde\Lambda}}
\def\La{{\Lambda}}
\def\bfj{{\mathbf j}}
\def\SnrQ{{\sS^\jmath(n,r)_{\mathbb Q(v)}}}
\def\SnQ{{\sS^\jmath(n)_{\mathbb Q(v)}}}
\def\SnrZ{{\sS^\jmath(n,r)}}

\def\SrrZ{{\sS^\jmath(r,r)}}
\def\Xinr{{\Xi_{2n+1,2r+1}}}
\def\Xirr{{\Xi_{2r+1,2r+1}}}
\def\Xin{{\Xi_{2n+1}}}
\def\Xinrl{{\Xi_{2n+1,2r+1}^{0\diag}}}
\def\Xinl{{\Xi_{2n+1}^{0\diag}}}

\def\fkm{{\mathfrak m}}
\def\fkmAl{{\mathfrak m^{A,\mathbf 0}}}
\def\bfl{{\mathbf 0}}
\def\fkmAj{{\mathfrak m^{A,\mathbf j}}}
\def\bfe{{\mathbf e}}
\def\bfUjn{{\mathbf U^\jmath(n)}}
\def\UjnZ{{U^\jmath(n)_\sZ}}
\def\Om{{\Omega}}
\def\om{{\omega}}
\def\bsi{{\boldsymbol i}}

\def\im{{\text{\rm im}}}
\def\lra{{\longrightarrow}}
\def\AjnZ{{\mathfrak{A}^\jmath(n)_\sZ}}
\def\Inr{{I(2n+1,r)}}

\def\td{{\widetilde d}}
\def\bsSjnr{{\boldsymbol{\sS}^\jmath(n,r)}}
\def\bsS{{\boldsymbol{\sS}}}
\begin{document}

\title{A new realisation of the $i$-quantum group $\bfU^\jmath(n)$}
\date{\today}

\author{Jie Du and Yadi Wu}

\address{J. D., School of Mathematics and Statistics,
University of New South Wales, Sydney NSW 2052, Australia}
\email{j.du@unsw.edu.au}
\address{Y. W., Department of Mathematics, Tongji University, Shanghai, 200092, China}
\email{1710062@tongji.edu.cn}


\date{\today}

\thanks{The paper was written while the second author was visiting the University of New South Wales as a practicum student for a year. She would like to thank UNSW for the hospitality and the China Scholarship Council for the financial support.}

\begin{abstract}We follow the approach developed in \cite{BLM} and modified in \cite{DF} to investigate a new realisation for the $i$-quantum groups $\bfUjn$, building on the multiplication formulas discovered in \cite[Lem.~3.2]{BKLW}. This allows us to present $\bfUjn$ via a basis and multiplication formulas by generators. We also establish a surjective algebra homomorphism from a Lusztig type form of $\bfUjn$ to integral $q$-Schur algebras of type $B$. Thus, base changes allow us to relate representations of the $i$-quantum hyperalgebras of $\bfUjn$ to representations of finite orthogonal groups of odd degree in non-defining characteristics. This generalises part of Dipper--James' type $A$ theory to the type $B$ case.
\end{abstract}

\maketitle

\tableofcontents

\section{Introduction}
Arising from representations of finite groups of Lie type, Iwahori--Hecke algebras play an important role in the study of unipotent principal blocks. In late 1980s, Dipper--James \cite{DJ89} introduced $q$-Schur algebras to study the representations of finite general linear groups in non-defining characteristics. In an entirely different context, these algebras appeared earlier in the study of the Schur--Weyl duality for a quantum linear group or quantum $\mathfrak{gl}_n$. Thus, $q$-Schur algebras naturally link representations of quantum $\mathfrak{gl}_n$ with those of finite general linear groups. It is natural to expect that such a connection extends to finite orthogonal and symplectic groups. However, example calculations on characters showed that this is not the case. In fact, from the invariant theory point of view, it is the Brauer algebra (or the BMW-algebra in the quantum case), instead of a group (or Hecke) algebra, that is involved in the Schur--Weyl duality.

Recently, in their study of canonical bases for quantum symmetric pairs, Bao and Wang \cite{BW} introduced certain co-ideal subalgebras $U^\jmath$ and $U^\imath$ of quantum linear groups, whose corresponding quantum symmetric pairs in \cite{Le03} are of type AIII. We follow \cite{CLW} to call them $i$-quantum groups. Bao--Wang further proved that these $i$-quantum groups pair with the Hecke algebras of type $B$ in a Schur--Weyl duality, where two types of Hecke endomorphism algebras or $q$-Schur algebras play the bridging role. More interestingly, as revealed in \cite{BKLW}, these $q$-Schur algebras arise naturally from finite orthogonal groups of odd degree or finite symplectic groups. 

With an entirely different motivation, the structure and representations of general Hecke endomorphism algebras were investigated by B. Parshall, L. Scott, and the first author
over twenty years ago. They made a stratification conjecture about their structure. A slightly modified version of this conjecture has recently been proved in \cite{DPS} and applications to representations of finite groups of Lie type have been obtained. 
In order to extend these applications to $i$-quantum groups, we need to lift Bao-Wang's duality to the integral level, and hence, to the roots-of-unity level. This is the aim of the current paper.

Building on the work of Bao et al \cite{BKLW}, we will define an algebra homomorphism $\phi^\jmath$ from the $i$-quantum groups $\bfUjn$  into the direct product of the corresponding $q$-Schur algebras $\SnrQ$ of type $B$. We mainly focus on the determination of the homomorphic image of $\phi^\jmath$ in terms of a BLM type basis $\{A(\bfj)\}_{A,\bfj}$ (cf. \cite{BLM}). This then allows us to investigate the Lusztig type forms and their associated hyperalgebras for the $i$-quantum groups through a certain monomial basis. We thus establish an algebra epimorphism from the integral $i$-quantum group to the integral $q$-Schur algebras of type $B$. Note that, with this epimorphism, the representation category of $i$-$q$-Schur algebras becomes a full subcategory of that of the $i$-quantum group (or $i$-quantum hyperalgebra). We further prove that the map $\phi^\jmath$ is injective. This gives a new realisation of the $i$-quantum group $\bfUjn$ in terms of the basis $\{A(\bfj)\}_{A,\bfj}$ together with explicit multiplication formulas by generators.

Just like the original realisation by Beilinson--Lusztig--MacPherson for quantum $\mathfrak{gl}_n$, this work laid down a foundation for a further study of the above mentioned $i$-quantum hyperalgebra.
On the other hand, it should be plausible that a similar realisation can be obtained for $i$-quantum groups $\bfU^\imath(n)$ although the triangular relation between two bases is a bit subtle in this case.
We will get these done in a forthcoming paper.

We organise the paper as follows. We first review in \S2 the definition of $i$-quantum groups $\bfUjn$ and their realisation as coideal sublagebra of the quantum linear group $\bfU(\mathfrak{gl}_{2n+1})$. In particular, the natural representation $\Om$ of $\bfU(\mathfrak{gl}_{2n+1})$ and its tensor product $\Om^{\otimes r}$ restrict to become $\bfUjn$-modules. In \S3, we introduce $q$-Schur algebra of type $B$ through finite orthogonal groups $\O_{2r+1}$ as well as (Iwahori--)Hecke algebras of type $B_r$. We will also review the Schur--Weyl duality discovered by Bao--Wang \cite{BW}.
Section 4 is quite parallel to \cite[\S5.2--\S5.4]{BLM}. Building on \cite[Lem. 3.2]{BKLW}, we derive certain multiplication formulas in $\SnrQ$. The structure constants in these formulas are independent of $r$. This allows to extend these formulas to get similar formulas in the direct product $\SnQ$ of $\SnrQ$ (Theorem \ref{MF}). Using a triangular relation in \cite[Thm. 3.10]{BKLW}, we prove in \S5 that the subspace $\Ajn$ spanned by all $A(\bfj)$ is a subalgebra of $\SnQ$ (Theorem \ref{th2}). The above mentioned homomorphism $\phi^\jmath$ has the image $\Ajn$ (Theorem \ref{UjtoA}). Using $\phi^\jmath$, we then lift Bao--Wang's epimorphism to the integral level (Theorem \ref{ISW}). Finally, in the last section, we prove that $\phi^\jmath$ is injective and thus, we
establish the new realisation (Theorem \ref{new realn}).

\medskip
\noindent
{\bf Some notations.} For a positive integer $a$, let
$$[1,a]=\{1,2,\ldots,a\},\quad [1,a)=\{1,2,\ldots,a-1\}.$$
Let $\sZ=\mathbb Z[v,v^{-1}]$ be the integral Laurent polynomial ring. For $n>0$, we set
$$[\![n]\!]=\frac{v^{2n}-1}{v^{2}-1},\quad [n]=\frac{v^n-v^{-n}}{v-v^{-1}}\quad and\quad [n]!=[1][2]\dots [n].$$
Set $[\![0]\!]=[0]=0$ and $[0]^!=1$. We also define, for $s,t\in\Z$ with $t>0$,
$$\left[\begin{matrix}s\\t\end{matrix}\right]=\prod_{i=1}^t\frac{v^{s-i+1}-v^{-(s-i+1)}}{v^i-v^{-i}},\quad
\left[\begin{matrix}K;s\\t\end{matrix}\right]=\prod_{i=1}^t\frac{Kv^{s-i+1}-K^{-1}v^{-(s-i+1)}}{v^i-v^{-i}},
$$
where $K$ is an element in a $\Q(v)$-algebra.

\section{The $i$-quantum group $\bfU^\jmath(n)$: a first realisation}

In the study of quantum symmetric pairs, Bao and Wang introduced the following quantum algebra $\bfU^\jmath(n)$, extracted from a quantum symmetric pair of type AIII in \cite[\S7]{Le03}. 
Here we follow the definition from \cite[\S4.3]{BKLW} and the $n$ indicates the rank of the $i$-quantum group.
\begin{defi}\label{iQG}
The algebra $\bfU^{\j}(n)$ is defined to be the associative algebra over $\Q(v)$ generated by $e_i$, $f_i$, $d_a$, $d_{a}^{-1}$, $i\in[1,n]$, $a\in[1,n+1]$ subject to the following relations: for $i,j\in[1,n]$, $a,b\in[1,n+1]$,
\begin{itemize}
\item[(iQG1)] $d_ad_a^{-1}=d_a^{-1}d_a=1, d_ad_b=d_bd_a$;
\item[(iQG2)] $d_ae_jd_a^{-1}=v^{\delta_{a,j}-\delta_{a,j+1}}e_j,$
                     $d_af_jd_a^{-1}=v^{-\delta_{a,j}+\delta_{a,j+1}}f_j$, if $a\leq n$;\\
                     $d_{n+1}e_jd_{n+1}^{-1}=v^{-2\delta_{n,j}}e_j,$
                     $d_{n+1}f_jd_{n+1}^{-1}=v^{2\delta_{n,j}}f_j$;
\item[(iQG3)] $e_if_j-f_je_i=\delta_{i,j}\frac{d_id_{i+1}^{-1}-d_i^{-1}d_{i+1}}{v-v^{-1}}$, if $i,j\neq n$;
\item[(iQG4)] $e_ie_j=e_je_i, f_if_j=f_jf_i,$ if $|i-j|>1$;
\item[(iQG5)] $ e_i^2e_j+e_je_i^2=[2]e_ie_je_i,$ 
 $ f_i^2f_j+f_jf_i^2=[2]f_if_jf_i$, if $|i-j|=1$;
\item[(iQG6)] $f_n^{2}e_n+e_nf_{n}^2=[2]\big(f_ne_nf_n-(vd_nd^{-1}_{n+1}+v^{-1}d_n^{-1}d_{n+1})f_n\big),$\\
$e_n^{2}f_n+f_ne_{n}^2=[2]\big(e_nf_ne_n-e_n(vd_nd^{-1}_{n+1}+v^{-1}d_n^{-1}d_{n+1})\big).$
\end{itemize}
\end{defi}

Note that the subalgebra generated by $e_i$, $f_i$, $d_a$, $d_{a}^{-1}$, $i\in[1,n)$, $a\in[1,n]$ is isomorphic to the quantum linear group $\bfU(\mathfrak{gl}_{n})$ in the following sense.

\begin{defi}\label{Ugln}
 The quantum linear group is a Hopf algebra $\bfU(\mathfrak{gl}_{N})$
over $ \mathbb Q(\up)$ with generators
$$E_{a}, F_{a},  K_{j}^{\pm  1},\;a\in[1,N), j\in[1,N],$$
and relations:
\begin{enumerate}
\item[(QG1)]
$ K_iK_j=K_jK_i,\ K_jK_j^{-1}=K_j^{-1}K_j=1; $
\item[(QG2)]
$ K_jE_{b}=\up^{\delta_{j,b}-\delta_{j,b+1}}E_{b}K_j,\
K_jF_{b}=\up^{-\delta_{j,b}+\delta_{j,b+1}}F_{b}K_j; $
\item[(QG3)]
$ [E_{a},
F_{b}]=\delta_{a,b}\frac{\wK_a-\wK_a^{-1}}{\up-\up^{-1}}, 
$ where ${\widetilde K}_{a}=K_aK_{a+1}^{-1}$;
\item[(QG4)]
$ E_{a}E_{b}=E_{b}E_{a}, \  F_{a}F_{b}=F_{b}F_{a},$ if $|a-b|>1 ;$
\item[(QG5)]
For $ a,b\in[1,N)$ with $|a-b|=1,$
\begin{equation}
\begin{aligned}\notag
&E_{a}^2E_{b}-(\up+\up^{-1})E_{a}E_{b}E_{a}+E_{b}E_{a}^2=0,\\
&F_{a}^2F_{b}-(\up+\up^{-1})F_{a}F_{b}F_{a}+F_{b}F_{a}^2=0.
\end{aligned}
\end{equation}
\end{enumerate}
Its comultiplication and counit are defined, respectively, by
\begin{equation}\notag
\aligned
\Delta:\bfU(\mathfrak{gl}_{N})&\longrightarrow \bfU(\mathfrak{gl}_{N})\otimes \bfU(\mathfrak{gl}_{N}),\\
E_a&\longmapsto 1\otimes E_a+E_a\otimes \wK_a,\\
F_a&\longmapsto F_a\otimes 1+\wK_a^{-1}\otimes F_a,\\
K_j&\longmapsto K_j\otimes K_j,
\endaligned
\;\;\;\;\text{ and }\;\;\;\;
\aligned
\epsilon:\bfU(\mathfrak{gl}_{N})&\longrightarrow \Q(v),\\
             E_a&\longmapsto 0,\\
             F_a&\longmapsto 0,\\
             K_j&\longmapsto 1.\endaligned
\end{equation}
\end{defi}

The first realisation of $\bfU^\j(n)$ is to embed it into the quantum group $\bfU(\mathfrak{gl}_{2n+1})$; see \cite[Prop. 4.5]{BKLW}.$^1$\footnote{$^1$We corrected a typo in the correspondence $d_{n+1}\mapsto vK_{n+1}^{-2}$ there.} See also \cite[Prop.~6.2]{BW} for the first version.

\begin{lemm}\label{iota}
There is an injective $\Q(v)$-algebra homomorphism $\iota:\bfU^\j(n)\to\bfU(\mathfrak{gl}_{2n+1})$ defined, for $i\in[1,n]$, by 
$$\aligned
d_i&\longmapsto K_i^{-1}K_{2n+2-i}^{-1}, \qquad d_{n+1}\longmapsto v^{-1}K_{n+1}^{-2},\\
e_i &\longmapsto F_i+\wK_i^{-1}E_{2n+1-i}, \quad f_i\longmapsto E_i\wK_{2n+1-i}+F_{2n+1-i}.\endaligned$$
Moreover, relative to the coalgebra structure, $\iota(\bfU^\j(n))$ is a coideal of $\bfU(\mathfrak{gl}_{2n+1})$.
\end{lemm}

Let $\Om=\Om_{2n+1}$ be the natural representation of $\bfU(\mathfrak{gl}_{2n+1})$ with a $\Q(v)$-basis $\{\om_1,\dots,\om_{2n+1}\}$ via the following actions:
\begin{equation}\label{NR}
E_h \om_{i}=\delta_{i,h+1}\om_h,\;F_{h}\om_{i}=\delta_{i,h}\om_{h+1},\; K_j\om_i=v^{\delta{i,j}}\om_i.
\end{equation}
Then $\Om^{\otimes r}$ becomes a $\bfU(\mathfrak{gl}_{2n+1})$-module via the actions:
$$E_h .\om_\bsi=\Delta^{(r-1)}(E_h)\om_\bsi,\;F_{h}.\om_\bsi=\Delta^{(r-1)}(F_h)\om_\bsi,\; K_j.\om_\bsi=\Delta^{(r-1)}(K_j)\om_\bsi,$$
where $\om_\bsi=\om_{i_1}\otimes\cdots\otimes\om_{i_r}$ for $\bsi=(i_1,\ldots,i_r)$, and
\begin{equation}\label{Dr-1}
\aligned
\Delta^{(r-1)}(K_j)&=\underbrace{K_j\otimes \cdots \otimes K_j}_r,\\
\Delta^{(r-1)}(E_h)&=\sum_{j=1}^r1\otimes\cdots\otimes1\otimes E_h\otimes\underbrace{\wK_h\otimes \cdots \otimes \wK_h}_{j-1},\\
\Delta^{(r-1)}(F_h)&=\sum_{j=1}^r\underbrace{\wK_h^{-1}\otimes \cdots \otimes \wK_h^{-1}}_{j-1}\otimes F_h\otimes1\otimes\cdots\otimes1.\\
\endaligned
\end{equation}
Thus, we obtain a $\mathbb Q(v)$-algebra homomorphism 
\begin{equation}\label{rho}
\rho_r:\bfU(\mathfrak{gl}_{2n+1})\to\End(\Om^{\otimes r}).
\end{equation} It is well-known that the image $\im(\rho_r)$ is the commutant subalgebra relative to a right action of the Hecke algebra of type $A$ (see Theorem \ref{Bao-Wang} below). This is called a $q$-Schur algebra
(of type $A$). We will soon see in next section that when we restrict $\rho_r$ to the subalgebra $\im(\iota)$:
$$\rho_r^\jmath:=\rho_r\circ\iota:\bfUjn\overset\iota\lra\bfU(\mathfrak{gl}_{2n+1})\overset{\rho_r}\lra\End(\Om^{\otimes r}),$$
the image $\im(\rho_r^\jmath)$ is the commutant subalgebra relative to a right action of the Hecke algebra of type $B$. This is called the $q$-Schur algebra of type $B$; see \cite[Thm. 6.27]{BW}.

\begin{lemm}\label{omega}The algebra $\bfU^{\j}(n)$ has an involutive automorphism 
$\omega$ defined by
    $$\omega(e_i)=f_i,\;\;\omega(f_i)=e_i,\;\;\omega(d_i)=d_i^{-1}\;\;(1\leq i\leq n),\omega(d_{n+1})=v^{-1}d_{n+1}^{-1}.$$
\end{lemm}
\begin{proof}This can be directly checked by the relations above or modified from \cite[Lem. 6.1]{BW}(1);
compare \cite[Lem~6.5(1)]{DDPW08}.
\end{proof}

\section{The $q$-Schur algebra of type $B$}

For any field $k$, let $\GL_n(k)$ be the genernal linear group over $k$ and consider the group isomorphism
\begin{equation}\notag\label{vartheta}
\vartheta:\GL_n(k)\longrightarrow \GL_n(k),\;x\longmapsto J^{-1}(x^t)^{-1}J,
\end{equation}
where $J$ has entries $J_{i,j}=1$ whenever $i+j=n+1$ and 0 otherwise.
The orthogonal group
$$\O_n(k):=\{x\in\GL_n(k)\mid J=x^tJx\}\quad(\text{where char}(k)\neq2) $$
is the fixed-point group of $\vartheta$.
Let  $G(q):=\O_{2r+1}(k)$ for $k=\mathbb F_q$, the finite field of $q$ elements. 

Let 
$$\Lambda(n+1,r)=\{\la=(\la_1,\ldots,\la_{n},\la_{n+1})\in \mathbb N^{n+1}\mid \la_1+\cdots+\la_{n+1}=r\}.$$ 

Define the bijection
\begin{equation}\label{wla}
\aligned
\widetilde{\ }:\Lambda(n+1,r)&\longrightarrow\wLa(n+1,r)\subseteq\Lambda(2n+1,2r+1),\\
                    \la&\longmapsto\widetilde\la:= (\la_{1},\ldots,\la_n,2\la_{n+1}+1,\la_n,\ldots,\la_{1}),\\
\endaligned
\end{equation}
where $\wLa(n+1,r)$ is the image of $\La(n+1,r)$ of the map.
For $\la\in \Lambda(n+1,r)$, let $P_{\widetilde\la}(q)$ be the standard parabolic subalgebra of $\GL_{2r+1}(\mathbb F_q)$ associated with $\widetilde\la$, consisting of upper quasi-triangular matrices with blocks of sizes $\widetilde\la_i$ on the diagonal.  Let
$$P_\la(q)=P_{\widetilde\la}(q)\cap G(q).$$
Then $G(q)$ acts on the set $P_\la(q)\backslash G(q)$ of left cosets $P_\la(q)g$ in $G(q)$. For any commutative ring $R$, this action induces a permutation representation over $R$ which is isomorphic to the induced representations $\text{Ind}_{P_{\la}(q)}^{G(q)}1_R$ of the trivial representation $1_R$ to $G(q)$ 
and define
\begin{equation}\label{UPB}
\mathcal E_{q,R}(n,r)=\End_{RG(q)}\bigg(\bigoplus_{\la\in\Lambda(n+1,r)}\text{Ind}_{P_{\la}(q)}^{G(q)}1_R\bigg)^{\text{op}}.
\end{equation}
This is called the $q$-Schur algebra of type $B$ (compare the type $A$ case in \cite{DJ89}).


 This algebra has the following interpretation of Hecke endomorphism algebra. Let $\sH(B_r)$ be the Hecke algebra over $\sZ=\mathbb Z[v,v^{-1}]$ associated with the Coxeter system $(W,S)$ of type $B_r$,
 where $S=\{s_1,\ldots,s_{r-1},s_r\}$ has the Dynkin diagram: 
 \begin{center}
\begin{tikzpicture}[scale=1.5]
\fill (0,0) circle (1.5pt);
\fill (1,0) circle (1.5pt);
\fill (2,0) circle (1.5pt);
\fill (4,0) circle (1.5pt);
\fill (5,0) circle (1.5pt);
  \draw (0,0) node[below] {$_1$} --
        (1,0) node[below] {$_2$} -- (2,0)node[below] {$_3$}--(2.5,0);
\draw[style=dashed](2.5,0)--(3.5,0);
\draw (3.5,0)--(4,0) node[below] {$_{r-1}$};
\draw (4,0.05) --
        (5,0.05);
\draw (4,-0.05) --
        (5,-0.05) node[below]  {$_r$};
\end{tikzpicture}
\end{center}
 Then it is generated by $T_i=T_{s_i}$ for $1\leq i\leq r$ subject to the relations:
 $$\aligned
 &T_i^2=(v^2-1)T_i+v^2, \forall i; \;\; T_iT_j=T_jT_i, |i-j|>2,\\
 &T_jT_{j+1}T_j=T_{j+1}T_jT_{j+1}, 1\leq j<r-1;\\
 &T_{r-1}T_rT_{r-1}T_r=T_rT_{r-1}T_rT_{r-1}.
 \endaligned$$
 It has a basis $\{T_w\}_{w\in W}$. The subalgebra generated by $T_1,\ldots,T_{r-1}$ is the Hecke algebra $\sH(\fS_r)$ associated with the symmetric group $\fS_r$.

 For 
$\la\in\Lambda(n+1,r)$, let 
$W_\la$ be the parabolic subgroup of $W$ generated by 
$$ S\backslash\{s_{\la_1+\cdots+\la_i}\mid i\in[1,n]\},$$
and let
$x_\la=\sum_{w\in W_\la}T_w$. The Hecke endomorphism $\sZ$-algebra:
\begin{equation}\label{Snr}
\SnrZ=\End_{\sH(B_r)}\big(\sT(n,r)\big),\quad\text{where }
\sT(n,r)=\bigoplus_{\la\in\Lambda(n+1,r)}x_\la\sH(B_r)
\end{equation}
is called the (generic) {\it $q$-Schur algebra of type $B$}.

In order to label the standard basis of $\SnrZ$ by matrices, 
we consider the graph automorphism of the symmetric group $\fS_{2r+1}=\fS_{[1,2r+1]}$: 
$$\sigma:\fS_{2r+1}\to\fS_{2r+1}, (i,j)\mapsto(2r+2-i,2r+2-j)\text{ for all }i,j\in[1,2r+1].$$
If $\theta\in\fS_{2r+1}$ denotes the permutation sending $i$ to $2r+2-i$, then, for any $\pi\in\fS_{2r+1}$,
$\sigma(\pi)=\theta\circ\pi\circ\theta$ and $\sigma(i,j)=(\theta(i),\theta(j))$.
Further, we may identify $W$ as the fixed-point subgroup $\fS_{2r+1}^\sigma$ of $\sigma$ with
$$s_i=(i,i+1)(2r+2-i,2r+1-i)\;\;(1\leq i<r),\quad s_r=(r,r+1)(r+1,r+2)(r,r+1).$$
For the parabolic subgroup $W_\la$, $\la\in\Lambda(n+1,r)$, if
$$
\widetilde \la=(\la_1,\ldots,\la_n,2\la_{n+1}+1,\la_n,\ldots,\la_1)\in\Lambda(2n+1,2r+1).
$$ as in \eqref{wla},
then $\sigma$ stabilises the Young subgroup $\fS_{\widetilde\la}$ and $W_\la=\fS_{\widetilde\la}^\sigma$ is the fixed-point subgroup.

\begin{lemm}
For $\la,\mu\in\Lambda(n+1,r)$ and $\td\in\fS_{2r+1}$, suppose the double coset $\fS_{\widetilde\la}\td\fS_{\widetilde\mu}$ defines a $(2n+1)\times(2n+1)$ matrix $A=(a_{i,j})$ over $\mathbb N$ whose entries sum to $2r+1$.
Then $\fS_{\widetilde\la}\td\fS_{\widetilde\mu}$ is stabilised by $\sigma$ if and only if $a_{i,j}=a_{N+1-i,N+1-j}$ for all $i,j\in[1,N]$ ($N=2n+1$).
\end{lemm}
\begin{proof}Let $R^\wla_i=\{\wla_1+\cdots\wla_{i-1}+1,\ldots,\wla_1+\cdots\wla_{i-1}+\wla_i\}$. Then
$a_{i,j}=|R^\wla_i\cap\td R^\wmu_j|$ and $\theta(R^\wla_i)=R^\wla_{N+1-i}$. Note that $\sigma$ stabilises $\fS_{\widetilde\la}$ and $\fS_{\widetilde\mu}$. Thus, $\sigma$ stabilises  $\fS_{\widetilde\la}\td\fS_{\widetilde\mu}$ (i.e., $\fS_{\widetilde\la}\td\fS_{\widetilde\mu}=\fS_{\widetilde\la}\sigma(\td)\fS_{\widetilde\mu}$) if and only if $|R^\wla_{N+1-i}\cap\sigma(\td) R^\wmu_{N+1-j}|=|R^\wla_{N+1-i}\cap \td R^\wmu_{N+1-j}|=a_{N+1-i,N+1-j}$ for all $i,j$. (See, e.g., \cite[Lem.~4.14]{DDPW08}.)
However,
$|R^\wla_{N+1-i}\cap\sigma(\td) R^\wmu_{N+1-j}|=|\theta (R^\wla_i\cap\td R^\wmu_j)|
=|R^\wla_i\cap\td R^\wmu_j|=a_{i,j}$.
\end{proof}
For $\la,\mu\in\Lambda(n+1,r)$, let $\sD_{\la,\mu}$ be the minimal length representatives of all double cosets $W_\la wW_\mu$. 

For $N=2n+1$, let 
\begin{equation}\label{Xi}
\aligned
\Xi_{2n+1}&=\big\{A=(a_{i,j})\in \text{Mat}_N(\mathbb{N}) \mid a_{i,j}=a_{N+1-i,N+1-j}, \forall i,j\in [1,N]\},\\
\Xi_{2n+1}^{0\diag}&=\big\{A-\diag(a_{1,1},a_{2,2},\ldots,a_{N,N})\mid A=(a_{i,j})\in\Xi_{2n+1}\},\\
\Xi_{2n+1,2r+1}&=\{A=(a_{i,j})\in\Xi_{2n+1}\mid |A|:=\sum_{i,j}a_{i,j}=2r+1\}.
\endaligned
\end{equation}
\begin{coro}There is a bijection 
$$\fkm: \{(\la,d,\mu)\mid \la,\mu\in\Lambda(n+1,r), d\in\sD_{\la,\mu}\}\longrightarrow \Xi_{2n+1,2r+1}.$$
\end{coro}
\begin{proof}The matrix $\fkm(\la,d,\mu)$ is the matrix associated with the double coset $\fS_{\widetilde\la}\widetilde d\fS_{\widetilde\mu}$.
The assertion follows from the lemma above.
\end{proof}

Recall the basis $\{e_A\mid A\in\Xi_{2n+1,2r+1}\}$ for $\mathcal E_{q,R}(n,r)$  introduced in \cite[\S2.3]{BKLW}.
\begin{theo} By specialising $v^2$ to $q=q1_R\in R$, there is an algebra isomorphism 
$$\xi: \mathcal E_{q,R}(n,r)\longrightarrow \sS^\jmath(n,r)_R:=\sS^\jmath(n,r)\otimes_\sZ R, \;e_A\longmapsto \xi^d_{\la,\mu},$$
where $\fkm(\la,d,\mu)=A$ and $\xi^d_{\la,\mu}$ is the $\sH(B_r)$-module homomorphism defined by
$$\xi^d_{\la,\mu}(x_\nu)=\delta_{\mu,\nu}\sum_{w\in W_\la dW_\mu}T_w.$$
\end{theo}
We will identify the two basis elements  $e_A=\xi^d_{\la,\mu}$ in the sequel.
\begin{remark} The proof is omitted as it is standard, extending Iwahori's original isomorphism $\End_{RG(q)}(\text{ind}_{P_{(1^r)}(q)}^{G(q)}1_R)\cong \sH(B_r)_R$. See \cite[Thm. 13.15]{DDPW08}. 
It should be noted that a more general result for all finite types can be found in \cite[Thm.~4.2]{LW}.
\end{remark}
 We record the fact that, if $D=\diag(\wla)\in\Xi_{2n+1,2r+1}$ for some $\la\in\La(n+1,r)$ is diagonal, then
\begin{equation}\label{idempotent}
e_De_A=\delta_{\wla,\ro(A)}e_A,\quad e_Ae_D=\delta_{\co(A),\wla}e_A
\end{equation} for all $A=(a_{i,j})\in\Xi_{2n+1,2r+1}$, where
\begin{equation*}
\aligned
\ro(A)&:=\big(\sum_{j}a_{1,j},\,\sum_{j}a_{2,j},\dots,\sum_{j}a_{2n+1,j}\big)\\
\co(A)&:=\big(\sum_{i}a_{i,1},\,\sum_{i}a_{i,2},\dots,\sum_{i}a_{i,2n+1}\big).
\endaligned
\end{equation*}

We end this section with Bao-Wang's Schur duality. We need to interpret the $\sH(B_r)$-module
$\sT(n,r)$ in terms of the tensor space $\Om^{\otimes r}$; see \cite{DPS}. Let 
\begin{equation}\label{Inr}
I(2n+1,r)=\{\bsi=(i_1,\ldots,i_r)\mid i_j\in[1,2n+1],\;\forall j\}.
\end{equation}
Then $\om_{\bsi}:=\om_{i_1}\otimes\cdots\otimes\om_{i_r}$, $\bsi\in \Inr$, form a basis for $\Om^{\otimes r}$.

Recall the action of $\sH(\fS_r)$ on $\Om^{\otimes r}$ defined in \cite[(14.6.4)]{DDPW08}: for $1\leq j< r$,
\begin{equation}\label{Sr action}\om_\bsi T_j=\begin{cases}
v\om_{\bsi s_j},&\text{ if }i_j<i_{j+1};\\
v^2\om_\bsi,&\text{ if }i_j=i_{i+1};\\
(v^2-1)\om_\bsi+v\om_{\bsi s_j},&\text{ if }i_j>i_{j+1}.\\
\end{cases}
\end{equation}
We extend the action to $\sH(B_r)$ by setting
\begin{equation}\label{Br action}
\om_\bsi T_r=\begin{cases}
v\om_{\bsi s_r},&\text{ if }i_j<r+1;\\
v^2\om_\bsi,&\text{ if }i_j=r+1;\\
(v^2-1)\om_\bsi+v\om_{\bsi s_r},&\text{ if }i_j>r+1.\\
\end{cases}
\end{equation}
Here $s_j=(j,j+1)$ and $\bsi s_j$ is the place permutation:
$$(i_1,\ldots,i_{r-1},i_r)s_j=\begin{cases}(i_1,\ldots,i_{j-1},i_{j+1},i_j,i_{j+2},\ldots,i_r),&\text{ if }j<r;\\
(i_1,\ldots,i_{j-1},i_{j},i_{j+1},\ldots,i_{r-1},i_{r+2}),&\text{ if }j=r,\end{cases}$$
where $i_{r+2}=N+1-i_{r-1}.$ 

As usual, we will call a surjective homomorphism an {\it epimorphism}. Let $\bsH(\fS_r)=\sH(\fS_r)_{\Q(v)}$, etc.
\begin{theo}\label{Bao-Wang}(1) 
The $\bfU(\mathfrak{gl}_{2n+1})$ action on $\Om^{\otimes r}$ commutes with the action of $\bsH(\fS_r)$ and the bimodule structure induces epimorphisms (cf. \eqref{rho})
$$\rho_r:\bfU(\mathfrak{gl}_{2n+1})\lra \End_{\bsH(\fS_r)}(\Om^{\otimes r}),\quad \rho_r^\vee:\bsH(\fS_r)\lra
\End_{\bfU(\mathfrak{gl}_{2n+1})}(\Om^{\otimes r}).$$


(2){\rm \cite[Lem.~5.3.8]{DPS}}There is an $\bsH(B_r)$-module isomorphism $\sT(n,r)_{\Q(v)}\cong \Om^{\otimes r}$. Hence,
$$\End_{\bsH(B_r)}(\Om^{\otimes r})\cong\SnrQ.$$

(3){\rm\cite[Thm.6.27]{BW}} The actions of $\bfUjn$ and $\bsH(B_r)$ commute and satisfy the double centraliser property as stated in (1). In particular, the algebra homomorphism $\rho_r$
restricts to an algebra epimorphism
\begin{equation}
\label{rhojr}\rho_r^\jmath=\rho_r\circ\iota:\bfUjn\lra\End_{\bsH(B_r)}(\Om^{\otimes r}).
\end{equation}
\end{theo}

\begin{remarks}\label{DanN}
(1) The $\Q(v)$-algebra $\sS(2n+1,r)_{\Q(v)}:=\End_{\sH(\fS_r)_{\Q(v)}}(\Om^{\otimes r})$ is known as the $q$-Schur algebra ($q=v^2$). By identifying $\SnrQ$ with $\End_{\sH(B_r)_{\Q(v)}}(\Om^{\otimes r})$ under the isomorphism in (2), we may regard $\SnrQ$ as a subalgebra of $\sS(2n+1,r)_{\Q(v)}$. Clearly, these relations hold at the integral level.

(2) The epimorphisms $\rho_r^\jmath$ induce an algebra homomorphism
$$\rho^\jmath:\bfUjn\lra\prod_{r\geq0}\End_{\sH(B_r)_{\Q(v)}}(\Om^{\otimes r}), \; u\longmapsto (\rho_r^\jmath(u))_{\r\geq0}.$$
One of the aims of this paper is to determine a basis for the image $\text{im}(\rho^\jmath)$ and to show that $\rho^\jmath$ induces an isomorphism $\bfUjn\cong \text{im}(\rho^\jmath)$.

(3) Lai, Nakano and Xiang considered the representation theory of $\SnrZ_k$ over a field $k$. In particular they realized the aforementioned algebra as the dual of the $r$th homogeneous component of the quotient of the coordinate algebra of the quantum matrix space by a right ideal that is also a coideal. This shows there is a natural polynomial representation theory (see \cite[Section 2.4]{LNX}). 

 Moreover, under a certain invertibility condition (i.e., the ``semisimple bottom'' condition in the sense of \cite{DR}), the structure and representations of these algebras including quasi-hereditariness and cellularity were investigated \cite[Sections 5-6]{LNX}. In turn, they obtained a concrete realization for the category $\sO$ of rational Cherednik algebras of type $B$ together with the Knizhnik--Zamolodchikov functor in terms of the module category of $\SnrZ_k$ and its corresponding Schur functor (see \cite[Section 8]{LNX}). 
\end{remarks}


\section{Some multiplication formulas}
We now derive some multiplication formulas and their associated stabilisation property in the $q$-Schur algebra of type $B$. This work is built on the formulas in \cite[Lem.~3.2]{BKLW}.

For $i,j\in[1,2n+1]$, let $E_{i,j}$ be the standard matrix units in $\text{Mat}_{2n+1}(\mathbb{N})$. Let
\begin{equation*}
E^{\theta}_{i,j}=E_{i,j}+E_{2n+2-i,2n+2-j}=E^\theta_{2n+2-i,2n+2-j}.
\end{equation*}
Note that $E^\theta_{n+1,n+1}=2E_{n+1,n+1}$. Let
\begin{eqnarray}\notag
\epsilon^{\theta}_{i,j}=
\begin{cases}
2 &\mbox{ if $i=j=n+1,$}\\[8pt]
1 &\,\,\mbox{otherwise},
\end{cases}\quad\text{and}\quad \bfe^\theta_i=\ro(E^\theta_{i,i}).
\end{eqnarray}
Then $\epsilon^{\theta}_{i,j}$ is the $(i,j)$-entry of $E^{\theta}_{i,j}$ and $\bfe^\theta_i=\bfe_i+\bfe_{2n+2-i}$, where $\mathbf e_i=(0,\ldots,0,\underset{(i)}1,0,\ldots0)\in\Z^{2n+1}$ form the standard basis for $\Z^{2n+1}.$

Recall the dimension $d(A)$ of the orbit $\sO_A$ and the dimension $r(A)$ of the image of $\sO_A$ under the first projection
(see \cite[(3.16)]{BKLW}).$^2$\footnote{${}^2$We thank Yiqiang Li for sending us this simplified version.} 
\begin{equation*}
d(A)-r(A)=\frac{1}{2}\big(\sum_{i\geq k,j<l}a_{ij}a_{kl}-\sum_{j<n+1\leq i}a_{ij}\big).\end{equation*}
We normalise the $\ZZ$-basis $\big\{e_{A}\mid A\in \Xi_{2n+1,2r+1}\big\}$ for $\sS^{\jmath}(n,r)$ by setting 
\begin{equation*}
[A]=v^{-d(A)+r(A)}e_{A}.
\end{equation*}
The following multiplication formulas are special cases of \cite[Thm. 3.7]{BKLW} by taking $R=1$. We can also derive them directly by \cite[Lem. 3.2]{BKLW}. For notational clarity, we extend the usual Kronecker delta $\delta_{i,j}$ to define
\begin{equation}\notag
\delta^\leq_{i,j}=\begin{cases}
1 &\mbox{ if $i\leq j,$}\\
0 &\,\,\mbox{if $i>j$}.
\end{cases}
\end{equation}
Let, for $A \in \Xi_{2n+1,2r+1}$ and $h\,\in\,[1,n]$,
\begin{equation}\label{beta}
\aligned
\b_p(A,h)&=\sum_{j\geq p}a_{h,j}-\sum_{j> p }a_{h+1,j}+\delta_{h,n}\delta^\leq_{p,n},\\
\b'_p(A,h)&=\sum_{j\leq p}a_{h+1,j}-\sum_{j< p}a_{h,j}.
\endaligned
\end{equation}
Recall $\wla$ in \eqref{wla} for $\la\in\La(n+1,r)$. We often use $\wla$ to denote the diagonal matrix $\diag(\wla)$ for simplicity. Thus, for any $A\in\Xinl$, $A+\wla$ really mean $A+\diag(\wla)$.
\begin{lemm}\label{key lem}
Suppose that $h\,\in\,[1,n]$, $\la\in\La(n+1,r-1)$, and $A \in \Xi_{2n+1,2r+1}$. The following multiplication formulas holds in $\SnrZ$:
\begin{itemize}
\item[(a)] 
$[E_{h,h+1}^{\theta}+\wla]\cdot[A]=\delta_{\bfe_{h+1}^\theta+\wla,\ro(A)}\displaystyle\sum_{p\in[1,2n+1]\atop a_{h+1,p}\geq \epsilon_{h+1,p}^{\theta}}v^{\beta_p(A,h)}\overline{[\![a_{h,p}+1]\!]}[A+E^{\theta}_{h,p}-E^{\theta}_{h+1,p}].$

\item[(b)] 
$[E_{h+1,h}^{\theta}+\wla]\cdot[A]=\delta_{\bfe_{h}^\theta+\wla,\ro(A)}\displaystyle\sum_{p\in[1,2n+1]\atop a_{h,p}\geq 1}v^{\b'_p(A,h)}\overline{[\![a_{h+1,p}+1]\!]}[A-E^{\theta}_{h,p}+E^{\theta}_{h+1,p}].$
\end{itemize}
\end{lemm}

\medskip
We now extend these formulas to a certain spanning set for $\sS^\jmath(n,r)$. Recall the notation $\Xi_{2n+1}^{0\diag}$ in \eqref{Xi}.

For $A\in\Xi_{2n+1}^{0\diag}$ and ${\bf j}=(j_1,j_2,\dots,j_N)\in \mathbb{Z}^{2n+1}$, define
\begin{equation}\label{Ajr}
A({\bf j},r)=\begin{cases}\displaystyle
\sum_{\la\in\Lambda(n+1,r-\frac{|A|}2)}v^{\widetilde\la\centerdot\bfj}[A+{\widetilde\la}],&\text{ if }|A|\leq 2r,\\
0,\qquad&\text{ if }|A|>2r,\end{cases}
\end{equation}
where $\wla\centerdot {\bf j}=\sum_{i=1}^{2n+1}\wla_i j_i$ with $\wla$ defined as in \eqref{wla}.
Note that 
$$\Big\{\wla\mid\la\in\La(n+1,r-\frac12{|A|})\Big\}=\{\mu\in \mathbb{N}^{2n+1}\mid A+{\mu}\in \Xi_{2n+1,2r+1}\}.$$
In particular, if $O$ denotes the zero matrix, $\mathbf e_i\in\Z^{2n+1}$ as above,
and ${\bf 0}=(0,\dots,0)\in\Z^{2n+1}$, we have $E^{\theta}_{h,h+1}({\bf 0},r)=\sum_{{\la}\in \Lambda(n+1,r-1)}[E^{\theta}_{h,h+1}+{\wla}]$ and
\begin{equation}\label{OE}
\aligned
&O(\bfe_i,r)=O(\bfe_{2n+2-i},r)=\begin{cases}\displaystyle\sum_{\la\in\La(n+1,r)}v^{\la_i}[\wla],&\text{ if }1\leq i\leq n;\\
\displaystyle\sum_{\la\in\La(n+1,r)}v^{2\la_{n+1}+1}[\wla],&\text{ if }i=n+1.\end{cases}
\endaligned
\end{equation}
For $1\leq h\leq n$, we put $\al_h=\bfe_{h}-\bfe_{h+1}$ and $\al^-_h=-\bfe_h-\bfe_{h+1}$. The following multiplication formulas are the type $B$ counterpart of \cite[Lem.~5.3]{BLM}.
\begin{theo}\label{th1}Maintain the notations introduced above. 
 For $N=2n+1$,  ${\bf j}=(j_1,j_2,\dots,j_{N})\in \mathbb{Z}^N$, $h\in[1,n]$, and $A=(a_{i,j})\in\Xizero$, the following multiplication formulas hold in $\sS^{\jmath}(n,r)$ for all $r\geq \frac{|A|}2$:
\begin{itemize}
\item[] \vspace{-3ex}
$$(1)\;\;{ O}({\bf j},r)A({\bf j'},r)=v^{\ro(A)\centerdot{\bf j}}A({\bf {j+j'}},r),\quad A({\bf j'},r){O}({\bf j},r)=v^{\co(A)\centerdot\bfj}A({\bf {j+j'}},r);\qquad\qquad\quad\;\;\quad$$
\item[]\vspace{-3ex}
\begin{equation*}
\aligned
(2)\;\;E^{\theta}_{h,h+1}({\bf 0},r)&\cdot A({\bf j},r)=\sum_{\substack{1\leq p<h\\ a_{h+1,p}\geq 1}}v^{\b_p (A,h)}\overline{[\![a_{h,p}+1]\!]}
(A+E^{\theta}_{h,p}-E^{\theta}_{h+1,p})({\bf j}+\alpha_{h},r)\\
&+\varepsilon \frac{v^{\b_h(A,h)-j_h-j_{N+1-h}-1}}{1-v^{-2}}\Big((A-E^{\theta}_{h+1,h})({\bf j}+\al _h,r)-(A-E^{\theta}_{h+1,h})({\bf j}+\al^-_h,r)\Big)\\
&+ v^{\b_{h+1}(A,h)+j_{h+1}+j_{N-h}}\overline{[\![a_{h,h+1}+1]\!]}(A+E^{\theta}_{h,h+1})({\bf j},r)\\
&+\sum_{\substack{h+1<p\leq N\\ a_{h+1,p}\geq 1}}v^{\b_p(A,h)}\overline{[\![a_{h,p}+1]\!]}
(A+E^{\theta}_{h,p}-E^{\theta}_{h+1,p})({\bf j},r)\\
\endaligned
\end{equation*}
where $\varepsilon=\delta^\leq_{1,a_{h+1,h}}$.
\item[]\vspace{-3ex}
\begin{equation*}
\aligned
(3)\, E^{\theta}_{h+1,h}(&{\bf 0},r)\cdot A({\bf j},r)=
\sum_{\substack{1\leq p<h\\ a_{h,p}\geq 1}}v^{\b'_p(A,h)}\overline{[\![a_{h+1,p}+1]\!]}
(A-E^{\theta}_{h,p}+E^{\theta}_{h+1,p})({\bf j},r)\\
&+ v^{\b'_h(A,h)+j_{h}+j_{N+1-h}}\overline{[\![a_{h+1,h}+1]\!]}(A+E^{\theta}_{h+1,h})({\bf j},r)\\
&+\varepsilon'\frac{v^{\b'_{h+1}(A,h)-j_{h+1}-j_{N-h}-1}}{1-v^{-2}}\bigg(\frac{(A-E^{\theta}_{h,h+1})({\bf j}-\al _h,r)}{v^{\delta_{h,n}}}-\frac{(A-E^{\theta}_{h,h+1})({\bf j}+\al^-_h,r)}{v^{-\delta_{h,n}}}\bigg)\\
&+\sum_{\substack{h+1<p\leq N\\ a_{h,p}\geq 1}}v^{\b'_p(A,h)}\overline{[\![a_{h+1,p}+1]\!]}
(A-E^{\theta}_{h,p}+E^{\theta}_{h+1,p})({\bf j}-\al_h,r),
\endaligned
\end{equation*}
where $\varepsilon'=\delta_{1,a_{h,h+1}}^\leq$.
\end{itemize}
\end{theo}

\begin{proof} Recall the map in \eqref{wla}. For $\la\in\La(n+1,r),\mu\in\La(n+1,r-\frac{|A|}2)$ with $A+\wmu\in\Xi_{2n+1,2r+1}$, \eqref{idempotent} implies
$[\wla][A+\wmu]=[A+\wmu]\iff \wla=\ro(A)+\wmu.$
Thus,
\begin{eqnarray*}
{ O}({\bf j},r)A({\bf j'},r)&=&\sum_{\la\in\Lambda(n+1,r)}\ \sum_{\mu\in \Lambda(n+1,r-\frac{|A|}2)}v^{\wla\centerdot{\bf j}+\wmu\centerdot{\bf j'}}[\wla][A+\wmu]\\
&=&v^{\ro(A)\cdot{\bf j}}\sum_{\mu\in \Lambda(n+1,r-\frac{|A|}2)}v^{\wmu\centerdot({\bf j}+{\bf j'})}[A+\wmu]\\
&=&v^{\ro(A)\cdot{\bf j}}\ A({\bf j}+{\bf j'},r).
\end{eqnarray*}
The proof for the second formula in (1) is similar. 

We now prove (2). By definition, we have,
\begin{eqnarray*}
E^{\theta}_{h,h+1}({\bf 0},r).A({\bf j},r)&=&\sum_{\la\in \Lambda(n+1,r-1)}\sum_{\mu \in \Lambda(n+1,r-\frac{|A|}2)}v^{\wmu\centerdot\bfj}[E^{\theta}_{h,h+1}+\widetilde{\l}][A+\wmu]\\
&=&\sum_{\mu\in \Lambda(n+1,r-\frac{|A|}2)}v^{\wmu\centerdot\bfj} [E^{\theta}_{h,h+1}+\ro(A)+\wmu-\co(E^{\theta}_{h,h+1})][A+\wmu]
\end{eqnarray*}
Let $A+\wmu=(a^{\mu}_{i,j})$. Then $a^{\mu}_{i,i}=a_{i,i}+\mu_{i}$ and $a^{\mu}_{i,j}=a_{i,j}$ for $i\neq j$.
For the number $\beta_p(A,h)$ in \eqref{beta}, we have
\begin{eqnarray}\label{A+mu}
\b_p(A+\wmu,h)=
\begin{cases}
\b_p(A,h)+\wmu_{h}-\wmu_{h+1} &\mbox{ if $p\leq h,$}\\[8pt]
\b_p(A,h) &\ \,\mbox{if $p\geq h+1$}.
\end{cases}
\end{eqnarray}
By Proposition \ref{key lem}(1) and noting $\co(E^{\theta}_{h,h+1})=\bfe_{h+1}+\bfe_{N-h}=\bfe^\theta_{h+1}$, we obtain that
\begin{equation*}
\aligned{}
[E^{\theta}_{h,h+1}+\ro(A)+&\wmu-\co(E^{\theta}_{h,h+1})][A+\wmu]
=[E^{\theta}_{h,h+1}+\ro(A)+\wmu-\bfe_{h+1}^{\theta})][A+\wmu]\\
&=\sum_{\substack{p\in[1,N]\\ a^{\mu}_{h+1,p}\geq \epsilon^{\theta}_{h+1,p}}}v^{\b_p(A+\wmu,h)}\overline{[\![a^{\mu}_{h,p}+1]\!]}[A+\wmu+E^{\theta}_{h,p}-E^{\theta}_{h+1,p}].
\endaligned
\end{equation*}
Hence, by \eqref{A+mu},
\begin{eqnarray*}
E^{\theta}_{h,h+1}({\bf 0},r).A({\bf j},r)
&=&\!\!\sum_{\mu\in \Lambda(n+1,r-\frac{|A|}2)}v^{\b_p(A+\wmu,h)+\wmu\centerdot\bfj}\sum_{\substack{p\in[1,N]\\a^{\mu}_{h+1,p}\geq \epsilon^{\theta}_{h+1,p} }}\overline{[\![a^{\mu}_{h,p}+1]\!]}[A+\wmu+E^{\theta}_{h,p}-E^{\theta}_{h+1,p}]\\
&=&\sum_{\substack{1\leq p<h\\a_{h+1,p}\geq 1}}v^{\b_p(A,h)}\overline{[\![a_{h,p}+1]\!]}\sum_{\mu}v^{\wmu\centerdot({\bf j}+\al_h)}[A+\wmu+E^{\theta}_{h,p}-E^{\theta}_{h+1,p}]\\
&+&\varepsilon\sum_{\mu}v^{\b_h(A+\wmu,h)+\wmu\centerdot\bfj}\overline{[\![a^{\mu}_{h,h}+1]\!]}[A+\wmu+E^{\theta}_{h,h}-E^{\theta}_{h+1,h}]\\
&+&\sum_{\mu:\wmu_{h+1}\geq \epsilon^{\theta}_{h+1,h+1}}v^{\b_{h+1}(A+\wmu,h)+\wmu\centerdot\bfj}\overline{[\![a_{h,h+1}+1]\!]}[A+\wmu+E^{\theta}_{h,h+1}-E^{\theta}_{h+1,h+1}]\\
&+&\sum_{\substack{h+1<p\leq N\\a_{h+1,p}\geq 1}}v^{\b_p(A,h)}\overline{[\![a_{h,p}+1]\!]}\sum_{\mu}v^{\wmu\centerdot\bfj}[A+\wmu+E^{\theta}_{h,p}-E^{\theta}_{h+1,p}],
\end{eqnarray*}
where $\mu$ runs over $\Lambda(n+1,r-\frac{|A|}2)$.
The first and last summations give the required form in (2). It remains to compute the second and third summations. If $a_{h+1,h}=a_{h+1,h}^\mu\geq1$, then $\varepsilon=1$ and
\begin{eqnarray*}
&&v^{\b_h(A+\wmu,h)+\wmu\centerdot\bfj}\overline{[\![a^{\mu}_{h,h}+1]\!]}[A+\wmu+E^{\theta}_{h,h}-E^{\theta}_{h+1,h}]\\
&=&v^{\b_h(A,h)+\wmu\centerdot({\bf j}+\al_h)}\overline{[\![\mu_h+1]\!]}[A-E^{\theta}_{h+1,h}+\wmu+\bfe_h^\theta]\\
&=&\frac{v^{\b_h(A,h)}}{1-v^{-2}}v^{\wmu\centerdot({\bf j}+\al_h)}\big(1-v^{-2(\wmu_h+1)}\big)[A-E^{\theta}_{h+1,h}+\wmu+\bfe_h^\theta]\\
&=&\frac{v^{\b_h(A,h)-j_h-j_{N+1-h}-1}}{1-v^{-2}}v^{\wmu\centerdot({\bf j}+\al_h)}v^{j_h+j_{N+1-h}+1}\big(1-v^{-2(\wmu_h+1)}\big)[A-E^{\theta}_{h+1,h}+\wmu+\bfe_h^\theta]\\
&=&\frac{v^{\b_h(A,h)-j_h-j_{N+1-h}-1}}{1-v^{-2}}\big(v^{(\wmu+\bfe_h^\theta)\centerdot({\bf j}+\bfe_h-\bfe_{h+1})}-v^{(\wmu+\bfe_h^\theta)\centerdot({\bf j}-\bfe_h-\bfe_{h+1})}\big)[A-E^{\theta}_{h+1,h}+\wmu+\bfe_h^\theta].
\end{eqnarray*}
Since $v^{\wla\centerdot({\bf j}+\bfe_h-\bfe_{h+1})}-v^{\wla\centerdot({\bf j}-\bfe_h-\bfe_{h+1})}=0$ whenever $\la_h=0$ (so $\wla_h=\wla_{N+1-h}=0$), it follows that
$$\aligned
\sum_{\mu\in\La(n+1,r-\frac{|A|}2)}&
\big(v^{(\wmu+\bfe_h^\theta)\centerdot({\bf j}+\bfe_h-\bfe_{h+1})}-v^{(\wmu+\bfe_h^\theta)\centerdot({\bf j}-\bfe_h-\bfe_{h+1})}\big)[A-E^{\theta}_{h+1,h}+\wmu+\bfe_h^\theta]\\
&=\sum_{\la\in\La(n+1,r-\frac{|A|}2+1)}\big(v^{\wla\centerdot({\bf j}+\alpha_h)}-v^{\wla\centerdot({\bf j}+\alpha_h^-)}\big)[A-E^{\theta}_{h+1,h}+\wla]\\
&=(A-E^{\theta}_{h+1,h})(\bfj+\alpha_h,r)-(A-E^{\theta}_{h+1,h})(\bfj+\alpha_h^-,r).
\endaligned$$
giving the second term in (2). Finally, for the third summation, we have
\begin{eqnarray*}
&&\sum_{\mu\in\La(n+1,r-\frac{|A|}2)\atop\wmu_{h+1}\geq \epsilon^{\theta}_{h+1,h+1}}v^{\b_{h+1}(A+\wmu,h)+\wmu\centerdot\bfj}\overline{[\![a_{h,h+1}+1]\!]}[A+\wmu+E^{\theta}_{h,h+1}-E^{\theta}_{h+1,h+1}]\\
&=&v^{\b_{h+1}(A,h)+j_{h+1}+j_{N-h}}\overline{[\![a_{h,h+1}+1]\!]}\sum_{\mu\atop\wmu_{h+1}\geq \epsilon^{\theta}_{h+1,h+1}}v^{(\wmu-e_{h+1}-e_{N-h})\centerdot\bfj}[A+E^{\theta}_{h,h+1}+\wmu-\bfe_{h+1}^\theta]\\
&=&v^{\b_{h+1}(A,h)+j_{h+1}+j_{N-h}}\overline{[\![a_{h,h+1}+1]\!]}(A+E^{\theta}_{h,h+1})({\bf j},r).
\end{eqnarray*}
Here, we have used an obvious bijection
$$\aligned
\Big\{\mu\in\La(n+1,r-\frac{|A|}2)\mid \mu_{h+1}\geq 1\Big\}&\lra\La(n+1,r-\frac{|A|}2-1)\\
(\mu_1,\ldots,\mu_{h},\mu_{h+1},\mu_{h+2},\ldots,\mu_{n+1})&\longmapsto(\mu_1,\ldots,\mu_{h},\mu_{h+1}-1,\mu_{h+2},\ldots,\mu_{n+1}).\endaligned$$ This proves (2). The proof of (3) is similar.\end{proof}

\begin{remark}\label{GLnMF}If one compares these multiplication formulas with those given in \cite[Lem.~5.3]{BLM} (or \cite[Thm. 13.27]{DDPW08}), they are very similar except the adjustments needed for the $h=n$ case. Of course, you may also see the difference arising from the symmetry of the matrices involved.
\end{remark}

Let $$\SnQ=\prod_{r\geq0}\SnrQ.$$
We will write the elements in $\SnQ$ as formal infinite series. 
Define, for $A\in\Xi_{2n+1}^{0\diag}$ and $\bfj\in\Z^{2n+1}$,
$$A({\bf j}):=\sum_{r\geq 0} A({\bf j},r)\in \SnQ.$$ 
For convenient use later, Theorem \ref{th1} is rewritten as follows.
\begin{theo}\label{MF}
For $N=2n+1$,  ${\bf j}=(j_1,j_2,\dots,j_{N})\in \mathbb{Z}^N$, $h\in[1,n]$, and $A=(a_{ij})\in\Xizero$, the following multiplication formulas hold in $\SnQ$:
\begin{itemize}
\item[] 
$(1)\;\;{ O}({\bf j})A({\bf j'})=v^{\ro(A)\centerdot{\bf j}}A({\bf {j+j'}}),\quad A({\bf j'}){O}({\bf j})=v^{\co(A)\centerdot\bfj}A({\bf {j+j'}});\qquad\qquad\quad\;\;$
\item[]\vspace{-3ex}
\begin{equation*}
\aligned
(2)\;\;E^{\theta}_{h,h+1}({\bf 0})&\cdot A({\bf j})=\sum_{\substack{1\leq p<h\\ a_{h+1,p}\geq 1}}v^{\b_p(A,h)}\overline{[\![a_{h,p}+1]\!]}
(A+E^{\theta}_{h,p}-E^{\theta}_{h+1,p})({\bf j}+\alpha_{h})\\
&+\delta^\leq_{1,a_{h+1,h}} \frac{v^{\b_h(A,h)-j_h-j_{N+1-h}-1}}{1-v^{-2}}\Big((A-E^{\theta}_{h+1,h})({\bf j}+\al _h)-(A-E^{\theta}_{h+1,h})({\bf j}+\al^-_h)\Big)\\
&+ v^{\b_{h+1}(A,h)+j_{h+1}+j_{N-h}}\overline{[\![a_{h,h+1}+1]\!]}(A+E^{\theta}_{h,h+1})({\bf j})\\
&+\sum_{\substack{h+1<p\leq N\\ a_{h+1,p}\geq 1}}v^{\b_p(A,h)}\overline{[\![a_{h,p}+1]\!]}
(A+E^{\theta}_{h,p}-E^{\theta}_{h+1,p})({\bf j}).\\
\endaligned
\end{equation*}
\smallskip
\item[]\vspace{-3ex}
\begin{equation*}
\aligned
(3)\, E^{\theta}_{h+1,h}&({\bf 0})\cdot A({\bf j})=
\sum_{\substack{1\leq p<h\\ a_{h,p}\geq 1}}v^{\b'_p(A,h)}\overline{[\![a_{h+1,p}+1]\!]}
(A-E^{\theta}_{h,p}+E^{\theta}_{h+1,p})({\bf j})\\
&+ v^{\b'_h(A,h)+j_{h}+j_{N+1-h}}\overline{[\![a_{h+1,h}+1]\!]}(A+E^{\theta}_{h+1,h})({\bf j})\\
&+\delta_{1,a_{h,h+1}}^\leq \frac{v^{\b'_{h+1}(A,h)-j_{h+1}-j_{N-h}-1}}{1-v^{-2}}\Big(\frac{(A-E^{\theta}_{h,h+1})({\bf j}-\al _h)}{v^{\delta_{h,n}}}-\frac{(A-E^{\theta}_{h,h+1})({\bf j}+\al^-_h)}{v^{-\delta_{h,n}}}\Big)\\
&+\sum_{\substack{h+1<p\leq N\\ a_{h,p}\geq 1}}v^{\b'_p(A,h)}\overline{[\![a_{h+1,p}+1]\!]}
(A-E^{\theta}_{h,p}+E^{\theta}_{h+1,p})({\bf j}-\al_h).
\endaligned
\end{equation*}
\end{itemize}
\end{theo}
\begin{proof}This is clear since the coefficients in the multiplication formulas in Theorem \ref{th1} are independent of $r$ for all $r\geq\frac12|A|$.
\end{proof}
\begin{coro}\label{cor1} For $h\in[1,n]$, we have in $\SnQ$
\begin{itemize}
\item[(1)] $E^{\theta}_{h,h+1}({\bf 0})^m=[m]^!  (mE^{\theta}_{h,h+1})({\bf 0});$
\item[(2)] $E^{\theta}_{h+1,h}({\bf 0})^m=[m]^!  (mE^{\theta}_{h+1,h})({\bf 0}).$
\end{itemize}
\end{coro}
\begin{proof} We only prove (1); the proof of (2) is similar.
For $A=mE^{\theta}_{h,h+1}$, the only non-zero entry in row $h+1$ is the diagonal entry and
$$\b_{h+1}(A,h)=\sum_{j\geq h+1}a_{h,j}-\sum_{j> h+1 }a_{h+1,j}+\delta_{h,n}\delta_{h+1,n}^\leq=m.$$ 
Thus, by Theorem \ref{MF}(2), 
$$E^{\theta}_{h,h+1}({\bf0})^2=v\overline{[\![1+1]\!]}(2E^{\theta}_{h,h+1})({\bf 0})=[2]^!(2E^{\theta}_{h,h+1})({\bf 0}).$$
Now the general case follows from an induction.
\end{proof}

\section{The subalgebra $\Ajn$}

We now prove that the subspace of $\SnQ$:
$$\Ajn=\text{span}\{A({\bf j})\mid A\in\Xinl,\,{\bf j}\in \mathbb{Z}^N\}$$
is indeed a subalgebra.

As in \cite{BKLW} or \cite[5.3]{BLM}, we define a preorder $\preceq$ on $\Xin$ as follows.
\begin{equation*}
A\preceq B \Longleftrightarrow
\sum_{r\leq i;s\geq j}a_{rs}\leq \sum_{r\leq i;s\geq j}b_{rs}, \mbox{ for all $1\leq i<j\leq 2n+1$}.
\end{equation*}
Clearly, $A\preceq B\Longleftrightarrow\sum_{r\geq i;s\leq j}a_{rs}\leq \sum_{r\geq i;s\leq j}b_{rs},\mbox{for all $i>j$}$.
We write $A\prec B$ if $A\preceq B$ and $B\not\preceq A$. 

 Let 
 $$\mathscr T_{2n+1}=\{(i,h,j)\mid 1\leq j\leq h<i\leq 2n+1\}.$$
  We order the set as in \cite[Thm.~3.10]{BKLW}:
\begin{equation}\label{order leq}
(i,h,j)\leq(i',h',j')\iff i<i'\text{ or }i=i',j<j'\text{ or }i=i',j=j', h>h'.
\end{equation}
This order modifies the order $\leq_i$ defined in \cite[(13.7.1)]{DDPW08}. For example, the first few elements in $(\mathscr T_{2n+1},\leq)$ are
\begin{equation}\label{seq}
\aligned
&(2,1,1),(3,2,1),(3,1,1),(3,2,2),(4,3,1),\ldots(4,3,2),\\
&(4,2,2),(4,3,3),\ldots, (N,N-1,N-1).
\endaligned
\end{equation}

For $A\in\Xinl$, let
\begin{equation}\label{MB}
\fkmAl:=\prod _{(i,h,j)\in(\mathscr T_{2n+1},\leq)}(a_{i,j}E^{\theta}_{h+1,h})({\bf 0}),
\end{equation}
where the product is taken with respect to the order $\leq$. Thus, by \eqref{seq},
 the leading term of the product $\fkmAl$ is $(a_{2,1}E^{\theta}_{2,1})({\bf 0})$ and the ending term is $(a_{N,N-1}E^{\theta}_{N,N-1})({\bf 0})$ ($N=2n+1$). 

\begin{lemm} \label{TR}For each $A\in\Xinl$, we have
$$\fkmAl=A(\bfl)+\sum_{B\in\Xinl,\bfj\in\Z^N\atop B\prec A}g_{A,B,\bfj}B(\bfj)=A(\bfl)+(\text{lower terms}).$$
\end{lemm}
\begin{proof}Repeatedly applying Theorem \ref{MF} yields
\begin{equation}\notag
\fkmAl=\sum_{\substack{B\in {\Xinl}\\{\bf j}\in \mathbb{Z}^N}}g_{A,B,{\bf j}}B({\bf j})
\end{equation}
It suffices to prove that $g_{A,A,{\bf 0}}=1$ and $B \prec A$ whenever $g_{A,B,\,{\bf j}}\neq 0$. Consider the $r$-th component of $\fkmAl$:
\begin{eqnarray*}
\pi_r(\fkmAl)&=&\prod _{1\leq j\leq h<i\leq N}^\leq(a_{i,j}E^{\theta}_{h+1,h})({\bf 0},r)\\
&=&\prod _{1\leq j\leq h<i\leq N}^\leq\sum_{\mu_{i,h,j}\in\La(n+1,r-a_{i,j})}[a_{i,j}E^{\theta}_{h+1,h}+\wmu_{i,h,j}]\\
&=&\sum_{\mu_{i,h,j}\in\La(n+1,r-a_{i,j})\atop (i,h,j)\in\mathscr T_{2n+1}}\prod _{1\leq j\leq h<i\leq N}^\leq[a_{i,j}E^{\theta}_{h+1,h}+\wmu_{i,h,j}].
\end{eqnarray*}
If such a product $\prod _{ j\leq h<i}^\leq[a_{i,j}E^{\theta}_{h+1,h}+\wmu_{i,h,j}]\neq0$, by \cite[Thm. 3.10]{BKLW}, there exists $\la\in\La(n+1,r-\frac{|A|}2)$ such that 
$$\prod _{ j\leq h<i}^\leq[a_{i,j}E^{\theta}_{h+1,h}+\wmu_{i,h,j}]=[A+\wla]+(\text{lower terms}).$$
Since 
$\ro(a_{2,1}E^\theta_{2,1}+\wmu_{2,1,1})=\ro(A+\wla)$ and $\co(a_{N,N-1}E^\theta_{N,N-1}+\wmu_{N,N-1,N-1})=\co(A+\wla)$, we have $\wla=\wmu_{2,1,1}+\bfe^\theta_2-\ro(A)$. Thus, with the notation in \cite[Thm.~3.10]{BKLW}, we have $D_{2,1,1}^\la:=\wmu_{2,1,1}$ and $D_{N,N-1,N-1}^\la:=\wmu_{N,N-1,N-1}$, and all other $D_{i,h,j}^\la=\wmu_{i,h,j}$ are completely determined by $A+\wla$.
Hence,
$$\aligned
\pi_r(\fkmAl)&=\sum_{\la\in\La(n+1,r-\frac{|A|}2)}
\prod _{1\leq j\leq h<i\leq N}^\leq[a_{i,j}E^{\theta}_{h+1,h}+D_{i,h,j}^\la]\\
&=\sum_{\la\in\La(n+1,r-\frac{|A|}2)}\Big([A+\wla]+(\text{lower terms})\Big)\\
&=A(\bfl,r)+\sum_{\substack{B\in\Xinl,\,B\prec A\\{\bf j}\in \mathbb{Z}^N}}g_{A,B,\,{\bf j}}B({\bf j},r),
\endaligned$$
as desired.
\end{proof}
\begin{theo}\label{th2}
The vector space $\Ajn=span\{A({\bf j})\mid A\in\Xinrl,\,{\bf j}\in \mathbb{Z}^{2n+1}\}$ of the algebra $\SnQ$ is a subalgebra which is generated by
$$ E^{\theta}_{h,h+1}({\bf 0}),\quad E^{\theta}_{h+1,h}({\bf 0}),\quad O({\bf \pm e_i}),$$
for all $1\leq h\leq n$ and $1\leq i\leq n+1$, and is presented by the multiplication formulas in Theorem \ref{MF}.
\end{theo}

\begin{proof}
Let $\fkA'$ be the subalgebra generated by $$ E^{\theta}_{h,h+1}({\bf 0}),\quad E^{\theta}_{h+1,h}({\bf 0}),\quad O({\bf \pm e_i}).$$ By Theorem \ref{MF}, we have $\fkA'\subseteq \Ajn$. We now prove $\Ajn\subseteq \fkA'$. We prove all $A(\bfj)\in\fkA'$ by induction on $\| A\|$, where
\begin{eqnarray*}
\|A\|=\sum_{1\leq i<j\leq N}\frac{(j-i)(j-i+1)}{2}(a_{ij}+a_{ji}).
\end{eqnarray*}

If $\| A\|=0$, then $A=O$, the zero matrix, and $A({\bf j})=O({\bf j})=\prod_{i=1}^{N}  O({\bf e_i})^{j_{i}}\in \fkA'$. Assume now $\| A\|> 0$ and $B({\bf j})\in \fkA'$ for $\| B\| < \| A\|$. Firstly, by Corollary \ref{cor1}, we have 
$$(a_{i,j}E^{\theta}_{h+1,h})({\bf 0})=\frac{E^{\theta}_{h+1,h}({\bf 0})^{a_{i,j}}}{[a_{i,j}]^! }\in\fkA'.$$
Thus, for all $A\in\Xinl, \bfj\in\Z^N$, $\fkmAl, O(\bfj)\fkmAl\in\fkA'$. By Lemma \ref{TR},
$$v^{-\ro(A)\centerdot\bfj}\ {O}({\bf j})\fkmAl=A({\bf {j}})+\sum_{B\in\Xinl,\bfj'\in\Z^N\atop B\prec A}g'_{A,B,\bfj'}B(\bfj').$$
Since $B\prec A$ implies that $\|B\|<\|A\|$, it follows from induction that all $B(\bfj')\in\fkA'$. Hence, $A(\bfj)\in\fkA'$.
\end{proof}
For $A\in\Xinl, \bfj\in\Z^N$, let $\fkmAj=O(\bfj)\fkmAl$.
\begin{coro} The $\Q(v)$-algebra $\Ajn$ has bases
$$\mathfrak B=\{A(\bfj)\mid A\in\Xinl, \bfj\in\Z^N\}\quad\text{and}\quad
\mathfrak M=\{\fkmAj\mid A\in\Xinl, \bfj\in\Z^N\}.$$
\end{coro}
\begin{proof}The assertion for $\mathfrak B$ is standard with an argument involving Vandermonde determinant.  The assertion for $\mathfrak M$ follows from the assertion for $\mathfrak B$ and the triangular relation in Lemma \ref{TR}.
\end{proof}

\begin{coro}\label{can proj} The canonical projection from $\SnQ$ onto $\SnrQ$ restricts to a $\Q(v)$-algebra epimorphism
$$\pi_r:\Ajn\to\SnrQ.$$
\end{coro}
\begin{proof}It suffice to prove that, for a fixed $A\in\Xinrl$,
$$\text{span}\{A(\bfj,r)\mid\bfj\in\Z^{2n+1}\}=\text{span}\{[A+\wla]\mid\la\in\La(n+1, r-\frac12|A|)\}.$$
This is clear from the definition of $A(\bfj,r)$ in \eqref{Ajr}.
\end{proof}
\section{Lifting Bao--Wang's Schur duality to the integral level}
In \cite[Thm.~6.27]{BW}, a $(\bfUjn, \bsH(B_r))$-duality via the tensor space $\Om^{\otimes r}$ is established; see Theorem~\ref{Bao-Wang}(3).  In this section, we will  define an algebra epimorphism  $\phi^\jmath_r:\bfUjn\to\SnrQ$ via the subalgebra $\Ajn$ and prove that $\phi^\jmath_r$ maps a Lusztig type form $U^\jmath(n)_\sZ$ of $\bfUjn$ onto the integral $q$-Schur algebra $\SnrZ$. We will compare $\phi^\jmath_r$ with the epimorphism $\rho^\jmath_r$ given in \eqref{rhojr} in next section.

\begin{theo}\label{UjtoA}
There is a $\mathbb Q(v)$-algebra epimorphism
$$ \phi^\jmath:\bfU^{\j}(n)\longrightarrow \Ajn$$
such that $e_{h}\mapsto E^{\theta}_{h,h+1}({\bf 0}),\;f_{h}\mapsto E^{\theta}_{h+1,h}({\bf 0}),\;d_{h}^{\pm1}\mapsto {O}({\bf \pm e_h})$, and $d^{\pm1}_{n+1}\mapsto v^{\mp1}{ O}({\pm \bfe_{n+1}})$ for all $1\leq h\leq n$.
\end{theo}

\begin{proof}
We must prove that the relations (iQG1)--(iQG6) in Definition \ref{iQG} are all satisfied for 
$$e_h=E^{\theta}_{h,h+1}({\bf 0}),\;f_h= E^{\theta}_{h+1,h}({\bf 0}),\;d_h =O({\bfe_h}),\; d_{n+1}=v^{-1}O(\bfe_{n+1}).$$ 
 Since relations (iQG1)--(iQG5) are more or less the defining relations for quantum $\mathfrak{gl}_n$, by Remark \ref{GLnMF}, the proof is almost the same as the proof of \cite[Thm.~13.33]{DDPW08}. We now prove (iQG6). We only check the first relation here:
\begin{equation}\label{iQG6}
f_n^{2}e_n+e_nf_{n}^2=[2]\big(f_ne_nf_n-(vd_nd^{-1}_{n+1}+v^{-1}d_n^{-1}d_{n+1})f_n\big).
\end{equation}

First, compute $f_n^2e_n=E_{n+1,n}^{\theta}({\bf 0})^2 E_{n,n+1}^{\theta}({\bf 0})$. By  Theorem \ref{MF}(3), we have
$$\aligned
f_ne_n&=E_{n+1,n}^{\theta}({\bf 0})E_{n,n+1}^{\theta}({\bf 0})=v^{\b'_n(E_{n,n+1}^{\theta},n)}(E_{n,n+1}^{\theta}+E_{n+1,n}^{\theta})({\bf 0})\\
&\quad+\frac{v^{\b'_{n+1}(E_{n,n+1}^{\theta},n)}}{1-v^{-2}}\Big(v^{-2}(E_{n,n+1}^{\theta}-E_{n,n+1}^{\theta})({-\al_{n}})-(E_{n,n+1}^{\theta}-E_{n,n+1}^{\theta})({\al^-_{n}})\Big)\\
&=(E_{n,n+1}^{\theta}+E_{n+1,n}^{\theta})({\bf 0})+\frac{v^{-2}}{1-v^{-2}}{O}({-\al_{n}})-\frac{O({\al^-_{n}})}{1-v^{-2}}
\endaligned
$$
Further, we see that,
\begin{equation}
\aligned\label{need below}
E_{n+1,n}^{\theta}({\bf 0})&(E_{n,n+1}^{\theta}+E_{n+1,n}^{\theta})({\bf 0})=v^{\b'_n(E_{n,n+1}^{\theta}+E_{n+1,n}^{\theta},n)}\overline{[\![2]\!]}(E_{n,n+1}^{\theta}+2E_{n+1,n}^{\theta})({\bf 0})\\
&\quad+\frac{v^{\b'_{n+1}(E_{n,n+1}^{\theta}+E_{n+1,n}^{\theta},n)}}{1-v^{-2}}\Big(v^{-2}E_{n+1,n}^{\theta}({-\al_{n}})-E_{n+1,n}^{\theta}({\al^-_{n}})\Big)\\
&=v\overline{[\![2]\!]}(E_{n,n+1}^{\theta}+2E_{n+1,n}^{\theta})({\bf 0})+\frac{v^{-1}}{1-v^{-2}}E_{n+1,n}^{\theta}({-\al_{n}})- \frac{v}{1-v^{-2}}E_{n+1,n}^{\theta}({ \al^-_{n}}),\\
\endaligned
\end{equation}
and
$E_{n+1,n}^{\theta}({\bf 0}){O}({-\al_{n}})=v^{-1}E_{n+1,n}^{\theta}({-\al_{n}}),$
$
E_{n+1,n}^{\theta}({\bf 0}){O}({\al^-_{n}})=v^{-1}E_{n+1,n}^{\theta}({\al^-_{n}}).
$
Thus, we have
$$\aligned
f_n^2e_n&=E_{n+1,n}^{\theta}({\bf 0})^2 E_{n,n+1}^{\theta}({\bf 0})\\
&=v\overline{[\![2]\!]}(E_{n,n+1}^{\theta}+2E_{n+1,n}^{\theta})({\bf 0})+\frac{v^{-1}}{1-v^{-2}}E_{n+1,n}^{\theta}({-\al_{n}})- \frac{v}{1-v^{-2}}E_{n+1,n}^{\theta}({\al^-_{n}})\\
&+\frac{v^{-3}}{1-v^{-2}}E_{n+1,n}^{\theta}({-\al_{n}})-\frac{v^{-1}}{1-v^{-2}}E_{n+1,n}^{\theta}({\al^-_{n}}),
\endaligned$$

Second, compute $e_nf_n^2=E_{n,n+1}^{\theta}({\bf 0})E^{\theta}_{n+1,n}({\bf 0})^2$.
Since 
$E^{\theta}_{n+1,n}({\bf 0})^2=[2]^!  (2E^{\theta}_{n+1,n})({\bf 0})$, by Corollary \ref{cor1},
it follows from Theorem \ref{MF}(2) that
$$\aligned
e_nf_n^2&=E_{n,n+1}^{\theta}({\bf 0})E^{\theta}_{n+1,n}({\bf 0})^2=[2]^!E_{n,n+1}^{\theta}({\bf 0})(2E_{n+1,n}^{\theta})({\bf 0}) \\
&=[2]^!\Big(\frac{v^{\b_n(2E_{n+1,n}^{\theta},n)-1}}{1-v^{-2}}\big[(2E_{n+1,n}^{\theta}-E_{n+1,n}^{\theta})({\al_{n}}) -(2E_{n+1,n}^{\theta}-E_{n+1,n}^{\theta})({\al^-_{n}})\big]\\
&\quad+v^{\b_{n+1} (2E_{n+1,n}^{\theta},n)}(2E_{n+1,n}^{\theta}+E_{n,n+1}^{\theta})({\bf 0})+v^{\b_{n+2}(2E_{n+1,n}^{\theta},n)}(E_{n+1,n}^{\theta}+E_{n,n+2}^{\theta})({\bf 0})\Big)\\
&=[2]^!(E_{n+1,n}^{\theta}+E_{n,n+2}^{\theta})({\bf 0})+[2]^!v^{-2}(2E_{n+1,n}^{\theta}+E_{n,n+1}^{\theta})({\bf 0})\\
&\quad +\ [2]^!\frac{v^{-2}}{1-v^{-2}}E_{n+1,n}^{\theta}({ \al_{n}})-[2]^!\frac{v^{-2}}{1-v^{-2}}E_{n+1,n}^{\theta}({\al^-_{n}})\\
\endaligned$$

Finally, compute the right hand side $f_ne_nf_n-(vd_nd_{n+1}^{-1}+v^{-1}d_n^{-1}d_{n+1})f_n$. Similar to $e_nf_n^2$, 
$$\aligned
e_nf_n&=E_{n,n+1}^{\theta}({\bf 0})E_{n+1,n}^{\theta}({\bf 0})\\
&=E_{n,n+2}^{\theta}({\bf 0})+v^{-1}(E_{n+1,n}^{\theta}+E_{n,n+1}^{\theta})({\bf 0})+\frac{v^{-1}}{1-v^{-2}}\big({O}({\al_n})-{O}({\al^-_n})\big),
\endaligned$$
Further, we have
$$\aligned
E_{n+1,n}^{\theta}({\bf 0})E_{n,n+2}^{\theta}({\bf 0})
&=v^{\b'_{n+2}(E_{n,n+2}^{\theta},n)}E_{n+1,n+2}^{\theta}({{-\al_{n}}})+v^{\b'_n(E_{n,n+2}^{\theta},n)}(E_{n+1,n}^{\theta}+E_{n,n+2}^{\theta})({\bf 0})\\
&=E_{n+1,n}^{\theta}({ {-\al_{n}}})+(E_{n+1,n}^{\theta}+E_{n,n+2}^{\theta})({\bf 0}),\\
\endaligned$$
This together with \eqref{need below} and Theorem \ref{MF}(1)
gives
$$\aligned
f_ne_nf_n&=E_{n+1,n}^{\theta}({\bf 0})\Big(E_{n,n+2}^{\theta}({\bf 0})+v^{-1}(E_{n+1,n}^{\theta}+E_{n,n+1}^{\theta})({\bf 0})+\frac{v^{-1}}{1-v^{-2}}\big({O}({\al_n})-{O}({\al^-_n})\big)\Big)\\
&= E_{n+1,n}^{\theta}({{-\al_{n}}})+(E_{n+1,n}^{\theta}+E_{n,n+2}^{\theta})({\bf 0})\\
&\quad+\overline{[\![2]\!]}(E_{n,n+1}^{\theta}+2E_{n+1,n}^{\theta})({\bf 0})
+\frac{1}{1-v^{-2}}\big(v^{-2}E_{n+1,n}^{\theta}({ -\al_{n}})-
E_{n+1,n}^{\theta}({\al^-_{n}})\big)\\
&\quad+\frac{1}{1-v^{-2}}\big(E_{n+1,n}^{\theta}({\al_n})-v^{-2}E_{n+1,n}^{\theta}({ \al^-_{n}})\big).
\endaligned
$$
Again, by Theorem \ref{MF}(1),
$$\aligned
vd_{n}d_{n+1}^{-1}f_{n}+v^{-1}d_{n}^{-1}d_{n+1}f_{n}&=v^2{O}({\bfe_n-\bfe_{n+1}})E_{n+1,n}^{\theta}({\bf 0})+v^{-2} {O}(-\bfe_n+{\bfe_{n+1}})E_{n+1,n}^{\theta}({\bf 0})\\
&=E_{n+1,n}^{\theta}({\al_n})+E_{n+1,n}^{\theta}({\bf -\al_{n}}).
\endaligned
$$
Now \eqref{iQG6} follows from coefficient equating of both sides:
\begin{center}
\begin{tabular}{|cc|}
\hline
{\bf Term}&{\bf Coefficient}\\ \hline
$(E_{n,n+1}^{\theta}+2E_{n+1,n}^{\theta})({\bf 0})$&$ v\overline{[\![2]\!]}+[2]^!v^{-2}=\overline{[\![2]\!]}[2]. $\\
$(E_{n+1,n}^{\theta}+E_{n,n+2}^{\theta})({\bf 0})$&$[2]^!=[2]$\\
$E_{n+1,n}^{\theta}({-\al_{n}})$&
$\frac{v^{-1}}{1-v^{-2}}+\frac{v^{-3}}{1-v^{-2}}=[2](1+\frac{v^{-2}}{1-v^{-2}}-1)$\\
$E_{n+1,n}^{\theta}({\al_n})$&
$ [2]^!\frac{v^{-2}}{1-v^{-2}}=[2](\frac{1}{1-v^{-2}}-1)$\\
$E_{n+1,n}^{\theta}({\al^-_{n}})$&
$-\frac{v}{1-v^{-2}}-\frac{v^{-1}}{1-v^{-2}}-[2]^!\frac{v^{-2}}{1-v^{-2}}=[2](-\frac{1}{1-v^{-2}}-\frac{v^{-2}}{1-v^{-2}})$\\\hline
\end{tabular}
\end{center}
which can be checked easily.\end{proof}

Recall the canonical projection map $\pi_r:\Ajn\to\SnrQ$ in Corollary \ref{can proj}. This together with the above result gives the following.

\begin{coro}\label{phijr}There is an algebra epimorphism 
$$\phi^\jmath_r=\pi_r\circ\phi^\jmath:\bfUjn\lra \SnrQ.$$
\end{coro}

\begin{remark}
We remark that the algebra homomorphimsm $\phi^\jmath_r$ has been established in \cite[Prop.~3.1]{BKLW} (compare \cite[Lem. 3.12]{BKLW}).$^3$\footnote{${}^3$The notation in \cite[Lem. 3.12]{BKLW} has been twisted by the involution $\omega$ in Lemma \ref{omega}.} However, the proof there is geometric via a geometric setting of $\SnrZ$. See \cite{LL} for an algebraic approach involving type $B$ Hecke algebras of two parameters.
\end{remark}

We now use a Lusztig type form $U^\jmath(n)_\sZ$ for $\bfUjn$  and show that $\phi^\jmath_r$ restricts to a $\sZ$-algebra epimorphism.

Define
\begin{eqnarray*}
\Bbbk_{i}:=
\begin{cases}
O(\bfe_i), &\mbox{ if $1\leq i\leq n,$}\\[8pt]
v^{-1}O(\bfe_{n+1}),&\,\,\mbox{if $i=n+1.$}
\end{cases}
\end{eqnarray*} and define $\Bbbk_{i,r}=\pi_r(\Bbbk_i)$.
Recall the notations $\Big[{\Bbbk_{i,r};0\atop \la_i}\Big]:=\prod_{j=1}^{\la_i}\frac{\Bbbk_{i,r}v^{-j+1}-\Bbbk^{-1}_{i,r}v^{j-1}}{v^{j}-v^{-j}}$ at the end of \S1 and $\wla$ in \eqref{wla}.
\begin{lemm}\label{v^2}
For any $\lambda\in\La(n+1,r)$, we have in $\SnrZ$
\begin{equation}\label{eqq-1} \notag
\prod_{i=1}^{n}\begin{bmatrix}\Bbbk_{i,r};0\\\l_{i}\end{bmatrix}\cdot\begin{bmatrix}\Bbbk_{n+1,r};0\\ \l_{n+1}\end{bmatrix}_{v^2}=[\wla].
\end{equation}
\end{lemm}
\begin{proof}We only outline the proof; missing details can be found in \cite[p.572]{DDPW08}.
For $1\leq i\leq n+1,\, 0\leq j\leq r$, let
\begin{equation*}
\d_{i}(j)=\sum_{\substack{\l\in \Lambda(n+1,r)\\ \l_{i}=j}}[\wla]
\end{equation*}
Then, by \eqref{OE}, we have $\sum^{r}_{j=0}v^{j}\,\d_{i}(j)=\Bbbk_{i,r}$ for $1\leq i\leq n$, and 
$$\Bbbk_{n+1,r}=v^{-1}\sum_{\l\in\La(n+1,r)}v^{2\l_{n+1}+1}[\wla]=\sum_{j=0}^rv^{2j}\d_{n+1}(j).$$
Thus, for $1\leq i\leq n$, 
$\begin{bmatrix}\Bbbk_{i,r};0\\\l_{i}\end{bmatrix}=\sum_{j\geq \l_{i}}\begin{bmatrix}{j}\\{\l_{i}}\end{bmatrix}\d_{i}(j)$, and
\begin{eqnarray*}
\begin{bmatrix}\Bbbk_{n+1,r};0\\\l_{n+1}\end{bmatrix}_{v^2}&=&\prod_{s=1}^{\l_{n+1}}\frac{\Bbbk_{n+1,r}v^{2(-s+1)}-k^{-1}_{n+1,r}v^{2(s-1)}}{v^{2s}-v^{-2s}}
=\prod_{s=1}^{\l_{n+1}}\big(\sum_{j=0}^{r}\frac{v^{2(j-s+1)}-v^{-2(j-s+1)}}{v^{2s}-v^{-2s}}\d_{n+1}(j)\big)\\
&=&\sum_{j=0}^{r}\big(\prod_{s=1}^{\l_{n+1}}\frac{v^{2(j-s+1)}-v^{-2(j-s+1)}}{v^{2s}-v^{-2s}}\big)\d_{n+1}(j)=\sum_{j\geq \l_{n+1}}\begin{bmatrix}{j}\\{\l_{n+1}}\end{bmatrix}_{v^2}\d_{n+1}(j).
\end{eqnarray*}
Hence,
\begin{eqnarray*}
\prod_{i=1}^{n}\begin{bmatrix}\Bbbk_{i,r};0\\
\l_{i}\end{bmatrix}\cdot\begin{bmatrix}\Bbbk_{n+1,r};0\\
\l_{n+1}\end{bmatrix}_{v^2}
=\sum_{j_{1}\geq \l_{1},\dots,j_{n+1}\geq \l_{n+1}}\bigg( \prod_{i=1}^{n}\begin{bmatrix}j_{i}\\{\l_i}\end{bmatrix}\cdot \begin{bmatrix}j_{n+1}\\{\l_{n+1}}\end{bmatrix}_{v^2}\bigg)\d_1(j_1)\dots\d_{n+1}(j_{n+1})=[\wla],
\end{eqnarray*}
as desired.
\end{proof}
Let $\UjnZ$ be the $\ZZ$-subalgebra generated by divided powers 
$$e_{i}^{(m)}:=\frac{e_i^{m}}{[m]^!},\; f_{i}^{(m)}:=\frac{f_i^{m}}{[m]^!},\;d_i,\;\begin{bmatrix}d_{i};0\\t\end{bmatrix}, d_{n+1}, \begin{bmatrix}d_{n+1};0\\t\end{bmatrix}_{v^2}$$
 for all $m,t\in\N$ and $1\leq i\leq n$.
 
 Recall the elements $\fkm^{A,\bfl}$ defined in \eqref{MB}. This is a product of $(a_{i,j}E^\theta_{h+1,h})(\bfl)$
for $1\leq j\leq h<i\leq 2n+1$ in the order defined in \eqref{order leq}. We now define $\sM^{A,\bfl}$
by replacing the factor $(a_{i,j}E^\theta_{h+1,h})(\bfl)$ for $h\leq n$ by $f_h^{(a_{i,j})}$ and replacing $(a_{i,j}E^\theta_{h+1,h})(\bfl)$ for $h> n$ by $e_{2n+1-h}^{(a_{i,j})}$.
Then $\sM^{A,\bfl}\in\UjnZ$.

For any $A\in\Xinl$ and $\l\in\N^{n+1}$, define elements in $\AjnZ$ and $\UjnZ$:
\begin{equation}\label{IB}
\fkm^{A,\l}=\bigg(\prod_{i=1}^{n}\begin{bmatrix}\Bbbk_{i};0\\\l_i\end{bmatrix}\cdot \begin{bmatrix}\Bbbk_{n+1};0\\\l_{n+1}\end{bmatrix}_{v^2}\bigg)\fkm^{A,\bfl},\quad
\sM^{A,\l}=\bigg(\prod_{i=1}^{n}\begin{bmatrix}d_{i};0\\\l_i\end{bmatrix}\cdot \begin{bmatrix}d_{n+1};0\\\l_{n+1}\end{bmatrix}_{v^2}\bigg)\sM^{A,\bfl}.
\end{equation}

\begin{theo}\label{ISW}The epimorphism $\phi^\jmath_r$ in Corollary \ref{phijr} induces by restriction
 a $\ZZ$-algebra epimorphism $\phi_r^\jmath: \UjnZ\longrightarrow \SnrZ$ such that
$$e_{h}^{(m)}\longmapsto (mE^\theta_{h,h+1})(\bfl,r),\; f_{h}^{(m)}\longmapsto (mE^\theta_{h+1,h})(\bfl,r),\;d_i\longmapsto \Bbbk_{i,r}.$$
\end{theo}
\begin{proof} Let  $\AjnZ$ be the $\sZ$-subalgebra generated by 
$$(mE^\theta_{i,i+1})(\bfl),\; (mE^\theta_{i+1,i})(\bfl), \;\Bbbk_i, \begin{bmatrix}\Bbbk_{i};0\\t\end{bmatrix},\;\Bbbk_{n+1}, \begin{bmatrix}\Bbbk_{n+1};0\\t\end{bmatrix}_{v^2}$$
 for all $m,t\in\N$ and $1\leq i\leq n$. Since the epimorphism $\phi^\jmath$ in Theorem \ref{UjtoA} sends the generators for $\UjnZ$ onto the generators of $\AjnZ$, it follows that restricting to $\UjnZ$ results in a $\sZ$-algebra epimorphism $\phi^\jmath:\UjnZ\to\AjnZ$. On the other hand, the canonical projection $\pi_r:\Ajn\to\SnrQ$ sends the generators of $\AjnZ$ to elements in $\SnrZ$.
Thus, $\phi_r^\jmath=\pi_r\circ\phi^\jmath$ defines a $\sZ$-algebra homomorphism
$$\phi^\jmath_r:\UjnZ\lra\SnrZ.$$
It remains to prove that $\phi^\jmath_r$ is surjective. 

For $A\in\Xin$, let $A'$ is obtained from $A$ by replacing the diagonal entries with zeros and let
$$\fkm^{(A)}:=\fkm^{A',\ro(A)}.$$ 
Then, by definition, $\phi^\jmath(\sM^{(A)})=\fkm^{(A)}$. Now apply $\pi_r$ to $\fkm^{(A)}$. For $\wla=\ro(A)$,
by Lemma \ref{v^2} and \cite[Thm. 3.10]{BKLW} (cf. the proof of Lemma \ref{TR}), we have
 $$\aligned 
 \pi_r(\fkm^{(A)})&=
 \pi_r\bigg(\prod_{i=1}^{n}\begin{bmatrix}\Bbbk_{i};0\\\l_i\end{bmatrix}\cdot \begin{bmatrix}\Bbbk_{n+1};0\\\l_{n+1}\end{bmatrix}_{v^2}\bigg)\pi_r(\fkm^{A',\bfl}),\\
 & =[\ro(A)]\displaystyle\prod _{(j,h,i)\in(\mathscr T_{2n+1},\leq)}(a_{i,j}E^{\theta}_{h+1,h})({\bf 0},r)\\
 &=[A]+(\text{lower terms}).
 \endaligned$$
Hence, the set $\{ \pi_r(\fkm^{(A)})\mid A\in\Xinr\}$ spans $\SnrZ$. This proves the surjectivity of $\phi_r^\jmath$.
\end{proof}

\begin{remarks}\label{IMB}
 (1) Due to the integral nature, we may specialise $\sZ$ to any commutative ring $k$ to get a $k$-algebra epimorphism $$\phi^\jmath_{r, k}:U^\jmath(n)_k\to\SnrZ_k.$$ Thus, if $k$ is a field, the representation category of $\SnrZ_k$ is a full subcategory of that of the hyperalgebra $U^\jmath(n)_k$ of $\bfUjn$. 
  In this way we link representations of the $i$-quantum groups (or $i$-quantum hyperalgebras) $U^\jmath(n)_k$ with those of the Hecke algebras of of type $B$. 

(2) Consider now representations of finite orthogonal groups $G(q)=\GL_n(\mathbb F_q)$ in non-defining characteristics. Those involves in \eqref{UPB} are related to the unipotent principal block which can be determined through the $q$-Schur algebra defined in \eqref{UPB}. Theorem \ref{ISW} extends further this relation to $i$-quantum groups.  It would be conceivable that much part of the classical Dipper--James theory can be generalised to finite orthogonal groups.

(3) We remark that a similar relation in terms of a certain $i$-quantum coordinate algebra is developed in \cite{LNX}. See Remark \ref{DanN}(3).

(4) The set (see \eqref{IB})
$$\big\{d_1^{\tau_1}\cdots d_{n+1}^{\tau_{n+1}}\sM^{A,\l}\mid A\in\Xinl,\l\in\N^{n+1},\tau_i\in\{0,1\}\big\}$$  
forms a $\mathbb Q(v)$-basis for $\bfUjn$ (cf. \cite[Lem.~6.47]{DDPW08}. It is reasonable to believe that the set forms a $\sZ$-basis for the integral form
$\UjnZ$.

(5) The integral form of the modified $i$-quantum group $\dot\bfU^\jmath(n)$ and its $i$-canonical basis have already been studied in \cite{BKLW} and its appendix by Bao--Li--Wang.
\end{remarks}

\section{A new realisation of $\bfU^\jmath(n)$}
The main purpose in this section is to prove that the $\Q(v)$-algebra homomorphism $\phi^\jmath$ in Theorem \ref{UjtoA} is in fact an isomorphism. Thus, we obtain a new realisation for the $i$-quantum group $\bfUjn$ by explicitly presenting its regular representation in terms of a basis and multiplication formulas in Theorem \ref{MF}, i.e., the matrix representations for generators.

We explain the idea of the proof as follows. Let $\bsS(2n+1,r)=\sS(2n+1,r)_{\Q(v)}$ and
$\bsSjnr=\SnrQ$ be the $q$-Schur algebras over $\Q(v)$ of type $A$ and $B$, respectively.
By Theorem \ref{Bao-Wang}(1), the algebra epimorphisms $\rho_r:\bfU(\mathfrak{gl}_{2n+1})\to\bsS(2n+1,r)$ induce a $\Q(v)$-algebra monomorphism (see, e.g. \cite{DG})
$$\rho=\prod_{r\geq0}\rho_r:\bfU(\mathfrak{gl}_{2n+1})\lra\bsS(2n+1):=\prod_{r\geq0}\bsS(2n+1,r).$$
It is known from Theorem \ref{Bao-Wang}(2) that $\bsS^\jmath(n,r)\subset \bsS(2n+1,r)$. Thus, by the embedding $\iota$ in Lemma \ref{iota}, both $\bfUjn$ and $\sS^\jmath(n)$ are subalgebras of $\bfU(\mathfrak{gl}_{2n+1})$ (via $\iota$) and $\bsS(2n+1)$, respectively. The idea of proving that $\phi^\jmath$ is injective is to show that $\phi^\jmath$ is the restriction $\rho^\jmath$ of $\rho$ to the image of $\iota$. In other words, we need to prove that $\phi^\jmath$ coincides with $\rho^\jmath$ via the isomorphism given in Theorem \ref{Bao-Wang}(2).
Thus, we must prove that the action of $\bfUjn$ on $\Om^{\otimes r}$ via $\iota$ coincides with the action of $\bfUjn$ on $\Om^{\otimes r}$ via
$\phi^\jmath$.

We need some preparation. There are two cases to consider.

If $n\geq r$, then the basis element $e_\emptyset:=[\diag(\emptyset)]\in\SnrZ$ is an idempotent, where
$$\emptyset=(\underbrace{1,\ldots,1}_r,\underbrace{0,\ldots,0}_{n-r},1,\underbrace{0\ldots,0}_{n-r},\underbrace{1,\ldots,1}_r)\in\N^{2n+1},$$
and $e_\emptyset\SnrZ e_\emptyset\cong \sH(B_r)$, $\SnrZ e_\emptyset\cong\sT(n,r)$ (see \eqref{Snr}). This gives an $\bsSjnr$-$\bsH(B_r)$-bimodule structure on $\bsSjnr e_{\emptyset}$. On the other hand, the tensor space $\Om^{\otimes r}$ is an $\bsSjnr$-$\bsH(B_r)$-bimodule via \eqref{Sr action} and \eqref{Br action}.
Moreover, there is an $\bsSjnr$-$\bsH(B_r)$-bimodule isomorphism
$$\eta_r:  \Om^{\otimes r}\lra \bsSjnr e_{\emptyset},\;\;\om_\bsi\longmapsto  [A_\bsi],$$
where $\bsi=(i_1,i_2,\dots,i_r)\in I(2n+1,r)$ and $A_\bsi=(a_{k,l})\in\Xinr$ defined for $N=2n+1$ by
\begin{equation}\label{Ai}
a_{k,l}=
\begin{cases}
\delta_{k,i_l}, &\text{ if }l\in[1,r];\\
\delta_{k,n+1} &\text{ if }l=n+1,\\
a_{N+1-k,N+1-l} &\text{ if } l\in[N+1-r, N]\\
0, &\text{ for the remaining columns.}\\
\end{cases}
\end{equation}
Note that $\co(A_\bsi)=(1^r,0^{n-r},1,0^{n-r},1^r)=\emptyset$.

We remark that the isomorphism $\eta_r$ is given in \cite{BKLW} (compare \cite[(2.9),(2.10)]{BKLW} with \cite[(4.11),(4.12)]{BKLW}, where $\tilde e_\bsi$ is the $[A_\bsi]$ here and ${\tt v}_\bsi$ is the $\om_\bsi$ here.

If $n<r$, then we may identify $\Xinr$ as a subset of $\Xirr$ via the following embedding:
\begin{equation}\label{Acirc}
\Xinr\lra \Xirr,\;\;A=\begin{pmatrix}X&|&Y\\ \text{---}&\cdot&\text{---}\\Y'&|&X'\end{pmatrix}
\longmapsto A^\circ=\begin{pmatrix}X&0&|&0&Y\\0&0&0&0&0\\ \text{---}&0&\cdot&0&\text{---}\\
0&0&0&0&0\\Y'&0&|&0&X'\end{pmatrix},
\end{equation}
where $X,X',Y,Y'$ are $n\times n$ matrices, ---$\cdot$--- and ${|\atop{\overset\cdot |}}$ represent the $n+1$st row and column of $A$, and the zeros in $A^\circ$ represent zero matrices of appropriate sizes. 
Thus, if $n<r$, we may regard $\SnrZ$ as a centraliser subalgebra of $\SrrZ$ via the induces embedding
$[A]\mapsto[A^\circ]$.

The embedding $[A]\mapsto[A^\circ]$ induces an embedding 
$$\wLa(n+1,r)=\{\ro(A)\mid A\in\Xinr\}\lra\wLa(r+1,r),\ro(A)\longmapsto\ro(A)^\circ:=\ro(A^\circ).$$
Let $f=\sum_{\l\in\La(n+1,r)}[\wla^\circ]$. Then there is an algebra isomorphism $f\SrrZ f\cong\SnrZ$.
This induces an $\bsSjnr$-$\bsH(B_r)$-bimodule isomorphism
$$\eta_r: \Om^{\otimes r} \lra f\bsS^\jmath(r,r) e_{\emptyset},\;\;\om_\bsi\longmapsto[A_\bsi].$$
Thus, as stated in Theorem \ref{Bao-Wang}(2), 
 $\eta_r$ induces an $\Q(v)$-algebra isomorphism in both cases:
$$\widetilde\eta_r:\End_{\bsH(B_r)}(\Om^{\otimes r})\lra\bsSjnr.$$
Recall the automorphism $\om$ in Lemma \ref{omega}.

\begin{theo}\label{new realn} The $\Q(v)$-algebra epimorphism $\phi_r^\jmath\circ \om:\bfUjn\lra\bsSjnr$
factors through the epimorphism $\rho_r^\jmath=\rho_r\circ\iota:\bfUjn\lra \End_{\bsH(B_r)}(\Om^{\otimes r})$, that is, $\widetilde\eta_r\circ\rho_r^\jmath=\phi_r^\jmath\circ\om$. Hence, the $\Q(v)$-algebra epimorphism
$$\phi^\jmath=\Pi_{r\geq0}\phi_r^\jmath:\bfUjn\lra\Ajn\subset\prod_{r\geq0}\bsS^\jmath(n,r)$$ is an isomorphism.
\end{theo}

\begin{proof} The commutative relations $\widetilde\eta_r\circ\rho_r^\jmath=\phi_r^\jmath\circ\om$ for all $r\geq0$ implies that the following  diagram commutes:
\begin{center}
\begin{tikzpicture}[scale=1.5]
\fill(-.3,1.5) node {$\bfUjn$};
\fill(1,1.7) node {$\rho^\jmath$};
\fill(1,.2) node {$\phi^\jmath$};
\fill(-.5,.8) node {$\om$};
\fill(3.2,.8) node {$\widetilde\eta$};
\fill(3.1,1.5) node {$\prod_{r\geq0}\text{End}_{\bsH(B_r)}(\Om^{\otimes r})$};
\fill(-.3,0) node {$\bfUjn$};
\fill(3.4,0) node {$\Ajn\subset\prod_{r\geq0}\bsSjnr$};
\draw[->](0.2,1.5) -- (1.8,1.5);
\draw[->](0.2,0) -- (1.8,0);
\draw[<-](-.3,.2) -- (-.3,1.3);
\draw[<-](3,.3) -- (3,1.3);
\end{tikzpicture}
\end{center}
where $\widetilde\eta=\prod_{r\geq0}\widetilde\eta_r$.
Since $\rho=\Pi_{r\geq0}\rho_r:\bfU(\mathfrak{gl}_{2n+1})\to\prod_{r\geq0}\text{End}_{\bsH(\fS_r)}(\Om^{\otimes r})$ is a monomorphism, it follows that the $\rho^\jmath$ is injective. The commutative diagram shows that $\phi^\jmath$ is injective. Hence, $\phi^\jmath$ is an isomorphism.

It remains to prove $\widetilde\eta_r\circ\rho_r^\jmath=\phi_r^\jmath\circ\om$.
We simply compare the actions of
 $\iota(d_i),\iota(e_h),\iota(f_h)$ on $\om_\bsi=\om_{i_1}\otimes \om_{i_2}\otimes\dots\otimes \om_{i_{r}}$, respectively, 
with the actions  on $[A_\bsi]$ of 
$$\pi_r\phi^\jmath\om(d_i)=O(-\bfe_i,r),\;\;  \pi_r\phi^\jmath\om(e_h)=E^\theta_{h+1,h}(\bfl,r),\;\; \pi_r\phi^\jmath\om(f_h)=E^\theta_{h,h+1}(\bfl,r),$$
where $A_\bsi$ is defined in \eqref{Ai}. By the embedding in \eqref{Acirc}, it suffices to assume that $n\geq r$.
Let $N=2n+1$ as usual.

For $1\leq h\leq n$ and $\bsi=(i_1,\ldots,i_r)\in I(N,r)$ as in \eqref{Inr}, we have by \eqref{NR},
\begin{eqnarray*}
\iota(d_h).(\om_{i_1}\otimes \om_{i_2}\otimes\dots\otimes \om_{i_{r}})&=&K^{-1}_{h}K^{-1}_{N+1-h}.\om_{i_1}\otimes \dots \otimes K^{-1}_{h}K^{-1}_{N+1-h}.\om_{i_r}\\
&=& v^{g_h}\om_{i_1}\otimes \om_{i_2}\otimes\dots\otimes \om_{i_{r}}
\end{eqnarray*}
where $g_h=-|\{k\mid 1\leq k\leq r,i_k=h\}|-|\{k\mid 1\leq k\leq r,i_k=N+1-h\}|$.\\

On the other hand, 
$$
O(-\bfe_h,r).[A_\bsi]=\sum_{\l\in\Lambda(n+1,r)}v^{-\l_h}[\wla]\cdot[A_\bsi]=v^{-\ro(A_\bsi)_h}[\ro(A_\bsi)][A_\bsi]=v^{-\ro(A_\bsi)_h}[A_\bsi],$$
where, by \eqref{Ai} and noting $h\neq n+1$,
\begin{eqnarray*}
\ro(A_\bsi)_h&=&|\{l\mid l\in[1,r],a_{h,l}=1=\delta_{h,i_l}\}\cup\{l\mid l\in[N-r+1,N],a_{h,l}=1\}|\\
&=&|\{l\mid l\in[1,r],i_l=h\}\cup\{ l\mid l\in[N-r+1,N],i_{N+1-l}=N+1-h\}|\\
&=&|\{k\mid 1\leq k\leq r,i_k=h\}|+|\{k\mid 1\leq k\leq r,i_{k}=N+1-h\}|=-g_h\\
\end{eqnarray*}
Now, for $h=n+1$,
\begin{eqnarray*}
\iota(d_{n+1}).(\om_{i_1}\otimes \om_{i_2}\otimes\dots\otimes \om_{i_{r}})&=&v^{-1}\underbrace{K^{-2}_{n+1}.\om_{i_1}\otimes \dots \otimes K^{-2}_{n+1}.\om_{i_r}}_r\\
&=& v^{g_{n+1}}\om_{i_1}\otimes \om_{i_2}\otimes\dots\otimes \om_{i_{r}}
\end{eqnarray*}
where $g_{n+1}=-1-2|\{k\mid 1\leq k\leq r,i_k=n+1\}|$. But
$$O(-\bfe_{n+1},r).[A_\bsi]=\sum_{\l\in\Lambda(n+1,r)}v^{-\wla_{n+1}}[\wla]\cdot[A_\bsi]=v^{-\ro(A_\bsi)_{n+1}}[\ro(A_\bsi)][A_\bsi]=v^{-\ro(A_\bsi)_{n+1}}[A_\bsi],$$
where, by \eqref{Ai},
\begin{eqnarray*}
\ro(A_\bsi)_{n+1}&=&2|\{k\mid 1\leq k\leq r,a_{n+1,k}=\delta_{n+1,i_k}=1\}|+a_{n+1,n+1}\\
&=&2|\{k\mid 1\leq k\leq r,i_k=n+1\}|+1=-g_{n+1}.
\end{eqnarray*}
This proves $\widetilde\eta_r\circ\rho_r^\jmath(d_i)=\phi_r^\jmath\circ\om(d_i)$ for all $i=1,2,\ldots,n+1$.

We now use the short notation $\om_{i_1}\om_{i_2}\cdots\om_{i_r}$ for a tensor product $\om_{i_1}\otimes \om_{i_2}\otimes\dots\otimes \om_{i_{r}}$. By Lemma \ref{iota},  \eqref{Dr-1}, and noting $\wK_i=K_iK_{i+1}^{-1}$,
we have
\begin{equation}\label{e_h}
\aligned
\iota(e_h).&(\om_{i_1} \om_{i_2}\dots \om_{i_{r}})=(F_h+\wK_h^{-1}E_{N-h}).(\om_{i_1} \om_{i_2}\dots \om_{i_{r}})\\
&\ =\ \sum_{l=1}^{r} \wK_h^{-1}\om_{i_1}\cdots \wK_{h}^{-1}\om_{i_{l-1}}\cdot F_h\om_{i_{l}} \cdot\om_{i_{l+1}}\dots \om_{i_r}\\
&\quad +\ \sum_{l=1}^{r}\wK^{-1}_{h}\om_{i_{1}}\cdots \wK_{h}^{-1}\om_{i_{l-1}} \wK^{-1}_{h}\cdot E_{N-h}\om_{i_{l}}\cdot\wK^{-1}_{h}\wK_{N-h}\om_{i_{l+1}}\dots \wK^{-1}_{h}\wK_{N-h}\om_{i_r}\\
&=\sum_{l=1\atop i_l=h}^{r} v^{f_1(l)}\om_{i_1}\dots \om_{i_{l-1}} \om_{h+1} \om_{i_{l+1}}\dots  \om_{i_r}\\
&\quad + \sum_{l=1\atop i_l=N-h+1}^{r}(v^{f_1(l)}\om_{i_{1}}\dots \om_{i_{l-1}}) \cdot(v^{\delta_{N-h,h+1}-\delta_{N-h,h}}\om_{N-h})\cdot (v^{f_2(l)}\om_{i_{l+1}}\dots \om_{i_r}),\\
\endaligned\end{equation}
where $f_1(l)=|\{k\mid 1\leq k<l,i_k=h+1\}|-|\{k\mid 1\leq k< l,i_k=h\}|$ and
$$\aligned
      f_2(l)&=-|\{k\mid l< k\leq r,i_k=N-h+1\}|+|\{k\mid l< k\leq r,i_k=N-h\}|\\
      &\quad+\ |\{k\mid l< k\leq r,i_k=h+1\}|-|\{k\mid l< k\leq r,i_k=h\}|.\endaligned$$
Thus, since $\delta_{N-h,h+1}=\delta_{h,n}$, $\delta_{N-h,h}=0$, and the summands in the second sum survive only when $i_l=N+1-h$, we may assume $i_l\neq h, h+1$ and
$$\aligned
f_1(l)+f_2(l)&+\delta_{N-h,h+1}-\delta_{N-h,h}\\
&=-|\{k\mid l< k\leq r,i_k=N-h+1\}|+|\{k\mid l< k\leq r,i_k=N-h\}|\\
      &\quad+\ |\{k\mid 1\leq k\leq r,i_k=h+1\}|-|\{k\mid 1\leq k\leq r,i_k=h\}|+\delta_{h,n}.
\endaligned$$
On the other hand, for the same $\bsi$ with $A_\bsi=(a_{k,l})$ and $\wmu=\ro(A_\bsi)-(\bfe_{h}+\bfe_{N-h+1})$, since $a_{h,l}=1$ forces $a_{h+1,l}=0$ and $a_{k,n+1}=0$ for all $n+1\neq k\in[1,N]$ by \eqref{Ai},
it follows that 
\begin{eqnarray*}
E^{\theta}_{h+1,h}({\bf 0},r).[A_\bsi]&=&\sum_{\l\in\Lambda(n+1,r-1)}[E^{\theta}_{h+1,h}+\wla]\cdot[A_\bsi]=[E^{\theta}_{h+1,h}+\wmu][A_\bsi]\\
&=&\sum_{l\in[1,N],a_{h,l}\geq 1}v^{\b'_l(A_\bsi,h)}[A_\bsi-E^{\theta}_{h,l}+E^{\theta}_{h+1,l}]\\
&=&\sum_{\substack{l\in[1,r]\\a_{h,l}=\delta_{h,i_l}=1}}v^{\b'_l(A_\bsi,h)}[A_\bsi-E^{\theta}_{h,l}+E^{\theta}_{h+1,l}]\\
&&\ +\ \sum_{\substack{l\in[N-r+1,N]\\a_{h,l}=\delta_{N+1-h,i_{N+1-l}}=1}}v^{\b'_l(A_\bsi,h)}[A_\bsi-E^{\theta}_{h,l}+E^{\theta}_{h+1,l}]\\
\end{eqnarray*}
where $\b'_l(A_\bsi,h)=\sum_{k\leq l}a_{h+1,k}-\sum_{k<l}a_{h,k}$. Since the bijection from $[N-r+1,N]$ to $[1,r]$  sending $l$ to $l'=N+1-l$ permutes the summands in the second sum, it follows that
\begin{eqnarray*}
E^{\theta}_{h+1,h}({\bf 0},r).[A_\bsi]&=&\sum_{\substack{l\in[1,r]\\i_l=h}}v^{\b'_l(A_\bsi,h)}[A_\bsi-E^{\theta}_{h,l}+E^{\theta}_{h+1,l}]\\
&&\ +\ \sum_{\substack{l\in[1,r]\\i_{l}=N+1-h}}v^{\b'_{N+1-l}(A_\bsi,h)}[A_\bsi-E^{\theta}_{h,N+1-l}+E^{\theta}_{h+1,N+1-l}]\\
\end{eqnarray*}
We now compare this with \eqref{e_h} via $\eta_r$.
Since, for $i_l=h$, $$\eta_r(\om_{i_1}\cdots\om_{i_{l-1}}\om_{h+1}\om_{i_{l+1}}\cdots\om_{i_r})=[A_\bsi-E^\theta_{h,l}+E^\theta_{h+1,l}]$$ and, for $i_l=N+1-h$, 
$$\eta_r(\om_{i_1}\cdots\om_{i_{l-1}}\om_{N-h}\om_{i_{l+1}}\cdots\om_{i_r})=[A_\bsi-E^\theta_{N+1-h,l}+E^\theta_{N-h,l}]=[A_\bsi-E^\theta_{h,N+1-l}+E^\theta_{h+1,N+1-l}],$$
 it remains to prove that
\begin{equation}\label{CC}
\aligned
\b'_l(A_\bsi,h)&=f_1(l),\qquad \text{ if }i_l=h;\\
\b'_{N+1-l}(A_\bsi,h)&=f_1(l)+f_2(l)+\delta_{N-h,h+1}-\delta_{N-h,h},\quad\text{ if }i_l=N+1-h.
\endaligned
\end{equation}
The first equation is clear since
\begin{eqnarray*}
\b'_l(A_\bsi,h)&=&\sum_{k\leq l}a_{h+1,k}-\sum_{k< l }a_{h,k}\\
&=&\,|\{k\mid 1\leq k\leq l,a_{\tilde{h}+1,k}=\delta_{h+1,i_{k}}=1\}|-
 |\{k\mid 1\leq k<l,a_{h,k}=\delta_{h,i_{k}}=1\}|\\
&=&\,|\{k\mid 1\leq k< l,i_{k}=h+1\}|-
 |\{k\mid 1\leq k<l,i_{k}=h\}|\;\;\;(\text{as }i_l=h)\\
 &=&f_1(l).
\end{eqnarray*}

For the second, since $l\in[1,r]\iff N+1-l\in[N-r+1,N]$, by \eqref{Ai},
$$\aligned
&\quad\b'_{N+1-l}(A_\bsi,h)=\sum_{1\leq k\leq N+1-l}a_{h+1,k}-\sum_{1\leq k< N+1-l }a_{h,k}\\
&=|\{k\mid 1\leq k\leq r,a_{h+1,k}=1\}|+\delta_{h,n}+|\{k\mid N-r+1\leq k\leq N+1-l,a_{h+1,k}=1\}|
\\
&\quad-
|\{k\mid 1\leq k\leq r,a_{h,k}=1\}|-|\{k\mid N-r+1\leq k< N+1-l,a_{h,k}=1\}|\\
&=|\{k\mid 1\leq k\leq r,i_{k}=h+1\}|+\delta_{h,n}+|\{k\mid N-r+1\leq k\leq N+1-l,i_{N+1-k}=N-h\}|
\\
&\quad-|\{k\mid 1\leq k\leq r,i_{k}=h\}|-|\{k\mid N-r+1\leq k< N+1-l,i_{N+1-k}=N+1-h\}|\\
&=
|\{k\mid 1\leq k\leq r,i_{k}=h+1\}|+\delta_{h,n}+|\{k\mid l\leq k\leq r,i_{k}=N-h\}|\\
&\quad\,-|\{k\mid 1\leq k\leq r,i_{k}=h\}|-|\{k\mid l< k\leq r,i_{k}=N+1-h\}|\\
&=f_1(l)+f_2(l)+\delta_{N-h,h+1}-\delta_{N-h,h},
\endaligned$$
proving \eqref{CC}. This proves $\widetilde\eta_r\circ\rho_r^\jmath(e_h)=\phi_r^\jmath\circ\om(e_h)$ for all $i=1,2,\ldots,n$. The proof of $\widetilde\eta_r\circ\rho_r^\jmath(f_h)=\phi_r^\jmath\circ\om(f_h)$ is similar
(and symmetric).
\end{proof}
By abuse of notation, let $A(\bfj)=(\phi^\jmath)^{-1}A(\bfj)$. Then
the basis $\{A(\bfj)\}_{A\in\Xinrl,\bfj\in\Z^{2n+1}}$ for $\Ajn$ gives rise to a new basis for $\bfUjn$. We now have a new presentation for $\bfUjn$ (cf. Theorem \ref{th2}).
\begin{coro} The $i$-quantum group $\bfUjn$ is a $\Q(v)$-algebra with basis 
$$\{A(\bfj)\mid A\in\Xinrl,\bfj\in\Z^{2n+1}\},$$ which has generators 
$$ E^{\theta}_{h,h+1}({\bf 0}),\; E^{\theta}_{h+1,h}({\bf 0}),\; {O}({\pm\bfe_i}),\; \text{ for all }1\leq h\leq n,1\leq i\leq n+1,$$ and relations 
{\rm(1)} ${O}({\pm\bfe_i})A(\bfj)=v^{\pm\ro(A)_i}A(\bfj\pm\bfe_i)$
together with {\rm(2)} and {\rm(3)} in Theorem \ref{MF}. 
\end{coro}
In other words, the multiplication formulas (1)--(3) give rise to the matrix form of the regular representation of $\bfUjn$.

\begin{remark}We expect to investigate applications of this new realisation for $\bfUjn$. For example, the existence of PBW type bases for $U^\jmath(n)_\sZ$ seems not clear. It is very plausible to construct such a basis by using the devided powers of ``root vectors'' $E_{i,j}^\theta(\bfl)$ for all $1\leq j<i\leq 2n+1$ and to establish a triangular relation with the integral monomial basis discussed in Remark \ref{IMB}(4).
\end{remark}

\bigskip
{\bf Acknowledgement.} The first author would like to thank  Weiqiang Wang, Yiqiang Li, and Zhaobin Fan
 for many helpful discussions and comments during the writing of the paper. He also thanks Dan Nakano for his questions and comments at the Charlottesville conference in March 2020, which motivated the current work.

 \end{document}